\documentclass[almostprime2]{article}
\usepackage{amsmath, amssymb, amsthm, delarray}

\newtheorem{thm}{Theorem}[section]
\newtheorem{lem}[thm]{Lemma}
\newtheorem{prop}[thm]{Proposition}

\theoremstyle{definition}

\newtheorem{rem}[thm]{Remark}

\numberwithin{equation}{section}
\theoremstyle{remark}

\begin{document}

\title{\textbf{Small gaps between the set of products of at most two primes}}

\author{Keiju Sono}

\allowdisplaybreaks

\date{}

\maketitle 
\noindent
\begin{abstract}
In this paper, we apply the methods of Maynard and Tao to the set of products of two distinct primes ($E_{2}$-numbers). We obtain several results on the distribution of $E_{2}$-numbers and primes. Among others, the result of Goldston,  Pintz,  Y{\i}ld{\i}r{\i}m and Graham \cite{GGPY2} on small gaps between $m$ consecutive $E_{2}$-numbers is improved.
\end{abstract}

\section{Introduction}
The famous twin prime conjecture asserts that there exist infinitely many prime numbers $p$ for which $p+2$ is also a prime, and this conjecture is widely believed to be true. More generally, about one hundred years ago, Hardy and Littlewood \cite{HL} conjectured the following, called the Hardy-Littlewood prime $k$-tuple conjecture. Let ${\cal H}=\{h_{1}, \cdots ,h_{k} \}$ be a set of  $k$ distinct non-negative integers. Then, the number of those $n$ below $N$ such that all of $n+h_{1}, \cdots ,n+h_{k}$ are primes will be asymptotically 
\[
\frac{N}{\log ^{k}N}\prod _{p}\left( 1-\frac{\nu _{p}({\cal H})}{p}\right)\left( 1-\frac{1}{p} \right)^{-k}
\]
provided that $\nu _{p}({\cal H}) <p$ for all primes $p$, where $\nu _{p}({\cal H})$ denotes the number of residue classes $\mathrm{mod}\; p$ covered by ${\cal H}$. In this case we say that the set ${\cal H}$ is admissible.  In particular, the twin prime conjecture is the case  $k=2$ and ${\cal H}=\{0, 2 \}$. Although this conjecture is still far from our reach, several remarkable results toward this have been established. For example,  in a celebrated paper \cite{C}, Chen proved that there exist infinitely many primes $p$ for which $p+2$ is either a prime or a product of two primes which are not necessarily distinct. 

Recently the studies toward the twin prime conjecture have produced further progress. In 2009, Goldston, Pintz and Y{\i}ld{\i}r{\i}m \cite{GPY1} proved that 
\begin{equation}
\label{GPY}
\liminf _{n\to \infty}\frac{p_{n+1}-p_{n}}{\log p_{n}}=0,
\end{equation}
where $p_{n}$ denotes the $n$-th prime. Their method is called the GPY sieve.  Moreover, they proved that if primes have the level of distribution $\theta$ for some $1/2 <\theta \leq 1$ (see the definition of $BV[\theta ,{\cal P}]$ below), then 
\begin{equation}
\label{GPY'}
\liminf _{n\to \infty}(p_{n+1}-p_{n})<\infty .
\end{equation}
The above assumption seems to be extremely difficult to prove, although it is known that the Bombieri-Vinogradov theorem assures that this is valid for $\theta \leq 1/2$. With $\theta =1/2$ one obtains (\ref{GPY}). The  case $\theta =1$ is called the Elliott-Halberstam conjecture (EH).  Several improvements have been made by the above three authors (see \cite{GPY2}, \cite{GPY3}, \cite{GPY4}). Among others, their best result on gaps between consecutive primes is that 
\begin{equation}
\label{GPY''}
\liminf _{n\to \infty}\frac{p_{n+1}-p_{n}}{\sqrt{\log p_{n}}(\log \log p_{n})^{2}}< \infty .
\end{equation}
Later, Pintz \cite{P} improved this result and obtained
\begin{equation}
\label{P}
\liminf _{n\to \infty}\frac{p_{n+1}-p_{n}}{(\log p_{n})^{\frac{3}{7}}(\log \log p_{n})^{\frac{4}{7}}} <\infty .
\end{equation}
(See also \cite{GPY5}.) In 2013, a stunning result was established by Zhang \cite{Z}. He obtained a stronger version of the Bombieri-Vinogradov theorem that is applicable when the moduli are free from large prime divisors, and using this, he proved that 
\begin{equation}
\label{Z}
\liminf _{n\to \infty}(p_{n+1}-p_{n})<7\times 10^{7},
\end{equation}
that is, there exist infinitely many consecutive primes for which the gap is at most $7\times 10^{7}$. The upper bound $7\times 10^{7}$ has been improved by several experts successively. In the Polymath8a paper \cite{Polymath1},  the right hand side of (\ref{Z}) was replaced by 4680. Slightly later, Maynard \cite{M1} and Tao (private communication with Maynard) invented a refinement  of the GPY sieve. In particular, Maynard proved that 
\begin{equation}
\label{M}
\liminf _{n\to \infty}(p_{n+1}-p_{n})\leq 600.
\end{equation}
They also proved the existence of the bounded gaps between $m$-consecutive primes for any fixed $m\geq 2$. One of the remarkable points is that their method is relatively quite simple, compared with Zhang's,  and it is very convenient to extend or generalize to other situations. The current world record of the small gaps between primes is accomplished by the Polymath project \cite{Polymath2}, in which the upper bound
\begin{equation}
\label{PM}
\liminf _{n\to \infty}(p_{n+1}-p_{n})\leq 246
\end{equation}
is obtained unconditionally. Moreover, it is proved that the right hand side of (\ref{PM}) may be replaced by $6$ if we assume a strong form of the Elliott-Halberstam conjecture and that is the limit of this method.

In this paper, we treat  the integers expressed by products of two distinct primes, called the $E_{2}$-numbers in \cite{GGPY1}, together with the prime numbers. In the papers \cite{GGPY1}, \cite{GGPY2}, Goldston, Graham, Pintz and Y{\i}ld{\i}r{\i}m investigated the distribution of $E_{2}$-numbers. We denote by $q_{n}$ the $n$-th $E_{2}$-numbers. That is, $q_{1}=6, q_{2}=10, q_{3}=14, q_{4}=15, \cdots$. Using the GPY sieve, they proved that 
\begin{equation}
\label{GGPY}
\liminf _{n\to \infty}(q_{n+1}-q_{n})\leq 6.
\end{equation}
Moreover, they proved that if the $E_{2}$-numbers have the level of distribution $\theta$ for some $0<\theta <1$, then for any sufficiently large $\rho \in \mathbf{N}$, 
\begin{equation}
\label{GGPY'}
\liminf _{n\to \infty}(q_{n+\rho }-q_{n})\leq \rho (1+o(1)) \exp \left(-\gamma +\frac{\rho}{2\theta} \right)
\end{equation}
holds. Later Thorne \cite{T} generalized their results to the set of products of $r$ distinct primes for any $r\geq 2$. He applied the result to some related problems in number theory, for example, divisibility of class numbers, nonvanishing of $L$-functions at the central point, and triviality of ranks of elliptic curves. 

The purpose of this paper is to apply the method of Maynard \cite{M1} and Tao to the distribution of $E_{2}$-numbers. Their multi-dimensional sieve enables us to  establish rather small gaps between consecutive several $E_{2}$-numbers. In particular, the estimate (\ref{GGPY'}) can be remarkably improved. We denote by ${\cal E}_{2}$ the set of all $E_{2}$-numbers.  We denote by ${\cal P}$ the set of all prime  numbers, and put ${\cal P}_{2}={\cal P}\cup {\cal E}_{2}$. For a sufficiently large natural number $N$, we define
\begin{eqnarray}
\beta (n):=
\left\{
\begin{array}{l}
1 \quad (n=p_{1}p_{2},\;  Y<p_{1}\leq N^{\frac{1}{2}}<p_{2}) \\
0 \quad (\mathrm{otherwise}),
\end{array}
\right.
\end{eqnarray}
where $Y=N^{\eta}$, $1\ll \eta <\frac{1}{4}$. Throughout this paper, the implicit constants might be dependent on this $\eta$. (We will not necessarily mention  this fact every time.) Next we define 
\[
\pi ^{\flat}(N):=\sharp \{p \in {\cal P}\; \vline \; N\leq p<2N \},
\]
\[
\pi ^{\flat}(N;q,a):=\sharp \{p \in {\cal P}\; \vline \; N\leq p<2N, p\equiv a\; (\mathrm{mod}\; q)\},
\]

\[
\pi _{\beta}(N):=\sum _{N<n\leq 2N}\beta (n), \quad \pi _{\beta ,q}(N):=\underset{(n,q)=1}{\sum _{N<n\leq 2N}}\beta (n),
\]
\[
\pi _{\beta}(N;q,a):=\underset{n\equiv a (\mathrm{mod}\; q)}{\sum _{N<n\leq 2N}}\beta (n).
\]
We write the following hypotheses:\\
\quad \\
\noindent
{\bf Hypothesis} 1 ($BV[\theta ,{\cal P}]$) \\
For any $\epsilon >0$, the estimate
\begin{equation}
\label{BVP}
\sum _{q\leq N^{\theta -\epsilon}} \mu ^{2}(q)\max _{(a,q)=1}\left| \pi ^{\flat}(N;q,a)-\frac{\pi ^{\flat}(N)}{\varphi (q)} \right|\ll _{A} \frac{N}{\log ^{A}N} \quad (N\to \infty)
\end{equation}
holds for any $A >0$. \\

\newpage

\noindent
{\bf Hypothesis} 2 ($BV[\theta ,{\cal E}_{2}]$) \\
We fix an arbitrary $0<\eta <\frac{1}{4}$ in the definition of the function $\beta$. For any $\epsilon >0$, the estimate
\begin{equation}
\label{BVE2}
\sum _{q\leq N^{\theta -\epsilon}}\mu ^{2}(q)\max _{(a,q)=1}\left| \pi _{\beta}(N;q,a)-\frac{\pi _{\beta ,q}(N)}{\varphi (q)} \right|\ll _{A} \frac{N}{\log ^{A}N} \quad (N\to \infty)
\end{equation}
holds for any $A >0$. \\
\quad \\
\noindent
We say that the set ${\cal P}$ (resp. ${\cal E}_{2}$) has level of distribution $\theta$ if $BV[\theta ,{\cal P}]$ (resp. $BV[\theta , {\cal E}_{2}]$) holds. The Bombieri-Vinogradov theorem  asserts that $BV[\theta , {\cal P}]$ is valid for $\theta =1/2$. Motohashi \cite{Mot} proved that $BV[\theta ,{\cal E}_{2}]$ also holds  for $\theta =1/2$. The Elliott-Halberstam conjecture asserts that $BV[\theta , {\cal P}]$ will be valid for $\theta =1$, and we expect that  $BV[\theta ,{\cal E}_{2}]$ will be valid for the same value. Hence we call $BV[1, {\cal P}]$ (resp. $BV[1, {\cal E}_{2}]$ ) the Elliott-Halberstam conjecture for ${\cal P}$ (resp. ${\cal E}_{2}$). 

The main theorems of this paper are as follows: 

\begin{thm}
Assume that the sets ${\cal P}$ and ${\cal E}_{2}$ have level of distribution $\theta >0$. Then, for any $\epsilon >0$, there exists $\rho _{\epsilon}>0$ such that for any integer $\rho >\rho _{\epsilon}$, the inequality 
\begin{equation}
\label{thm1}
\liminf _{n\to \infty}(q_{n+\rho}-q_{n})\leq \exp \left( \frac{(2+\epsilon )\rho}{3\theta \log \rho } \right)
\end{equation}
holds. In particular, unconditionally we have 
\begin{equation}
\label{thm1.5}
\liminf _{n\to \infty}(q_{n+\rho}-q_{n})\leq \exp \left( \frac{2(2+\epsilon )\rho}{3 \log \rho } \right)
\end{equation}
for any $\rho >\rho _{\epsilon}$.
\end{thm}

\begin{thm}
For any admissible set ${\cal H}=\{h_{1}, h_{2}, \cdots ,h_{6} \}$, there exist  infinitely many $n$ such that at least three of $n+h_{1}, n+h_{2}, \cdots ,n+h_{6}$ are in ${\cal P}_{2}$. 
\end{thm}
The set ${\cal H}=\{0,4,6,10,12,16 \}$ is an admissible set with six elements. Hence if we denote by $r_{n}$ the $n$-th element of ${\cal P}_{2}={\cal P}\cup {\cal E}_{2}$, unconditionally we have
\begin{equation}
\label{p2}
\liminf _{n\to \infty} (r_{n+2}-r_{n}) \leq 16.
\end{equation}
If we assume the Elliott-Halberstam conjecture for both ${\cal P}$ and ${\cal E}_{2}$, far stronger results can be obtained: 
\begin{thm}
Assume the Elliott-Halberstam conjecture for ${\cal P}$ and ${\cal E}_{2}$. Then, there exist infinitely many $n$ such that all of $n, n+2, n+6$ are in ${\cal P}_{2}$. In particular, 
\begin{equation}
\label{p2'}
\liminf _{n\to \infty} (r_{n+2}-r_{n}) \leq 6.
\end{equation}
\end{thm}
We note that Maynard \cite{M2} unconditionally proved that $n(n+2)(n+6)$ has at most seven prime factors infinitely often. Theorem 1.3 is regarded as a (conditional) improvement of his theorem. Finally,
\begin{thm}
Assume the Elliott-Halberstam conjecture for ${\cal P}$ and ${\cal E}_{2}$. Then we have 
\begin{equation}
\label{thm4}
\liminf _{n\to \infty}(q_{n+2}-q_{n}) \leq 12.
\end{equation}
\end{thm}

%%%%%%%%%%%%%%%%%%%%%%%%%%%%%%%%%%%%%%%%%%%%%%%%%%%%%%%%%%%%%%%%%%%%%%%%%%%%%%%%%%%%%%%%%%%%%%%%%%%%%%%%%%%%%%%%%%%%%%%%%%%%%%%%%%%%%%%%%%%%%%%%%%%%%%%%%%%%%%%%
\section{Notation and preparations for the proofs}
Let ${\cal H}=\{h_{1}, \cdots , h_{k}\}$ be an admissible set. Throughout this paper, we assume that the elements of ${\cal H}$ are bounded, that is, there exists a positive constant $C=C_{k}$ depending only on $k$ such that $h_{i}\leq C$ holds for $i=1, \cdots ,k$.  We denote by $\chi _{{\cal P}}$  the characteristic function of ${\cal P}$. We put 
\[
D_{0}=\log \log \log N,\quad  W=\prod _{p\leq D_{0}}p \ll (\log \log N)^{2}.
\]
We assume that both prime numbers and $E_{2}$-numbers have level of distribution $\theta$. By the Chinese remainder theorem, we can choose $\nu _{0}\in \mathbf{N}$ so that all of $\nu _{0}+h_{i}$ ($i=1, \cdots ,k$) are coprime to $W$. For a smooth function $F: \mathbf{R}^{k}\to \mathbf{R}$ supported in
\[
{\cal R}_{k}:=\{ (x_{1}, \cdots ,x_{k})\; | \; x_{1}, \cdots , x_{k}\geq 0, \sum _{i=1}^{k}x_{i}\leq 1 \},
\]
we put
\begin{equation}
\label{3}
\lambda _{d_{1},\cdots ,d_{k}}=\left( \prod _{i=1}^{k}\mu (d_{i})d_{i}\right) \underset{(r_{i},W)=1 (\forall i)}{\underset{d_{i}|r_{i} (\forall i)}{\sum _{r_{1}, \cdots ,r_{k}}}}\frac{\mu (\prod _{i=1}^{k}r_{i})^{2}}{\prod _{i=1}^{k}\varphi (r_{i})}F\left( \frac{\log r_{1}}{\log R}, \cdots ,\frac{\log r_{k}}{\log R} \right)
\end{equation}
if $(d_{1}, \cdots ,d_{k})$ satisfies the conditions that $\prod _{i=1}^{k}d_{i}$ is square-free, $\prod _{i=1}^{k}d_{i} < R$, and $(d_{i},W)=1$ for $i=1, \cdots ,k$, where $R=N^{\frac{\theta}{2}-\delta}$ and $\delta$ is a sufficiently small positive constant. If $(d_{1}, \cdots ,d_{k})$ does not satisfy at least one of these conditions, put $\lambda _{d_{1}, \cdots ,d_{k}}:=0$.  We define the weight $w_{n}$ by 
\begin{equation}
\label{2}
w_{n}=\left( \sum _{d_{i}|n+h_{i} (\forall i)}\lambda _{d_{1},\cdots ,d_{k}}\right)^{2}.
\end{equation}
To find small gaps between $E_{2}$-numbers, for a natural number $\rho$, we consider the sum 
\begin{equation}
\label{1}
S(N, \rho)=\underset{n\equiv \nu _{0} (\mathrm{mod}\; W)}{\sum _{N\leq n<2N}}\left( \sum _{m=1}^{k}\beta (n+h_{m})-\rho \right)w_{n}.
\end{equation}
If $S(N, \rho)$ becomes positive for any sufficiently large $N$, there exists $n\in [N, 2N)$ such that at least $\rho +1$ of $n+h_{1}, \cdots ,n+h_{k}$ are $E_{2}$-numbers. Hence one has
\[
\liminf_{n\to \infty}(q_{n+\rho}-q_{n})\leq \max _{1\leq i<j\leq k}|h_{j}-h_{i}|.
\]
 Similarly, to find small gaps between the set of primes and $E_{2}$-numbers, for a natural number $\rho$,  we consider the sum 
\begin{equation}
\label{1'}
S^{'}(N, \rho)=\underset{n\equiv \nu _{0} (\mathrm{mod}\; W)}{\sum _{N\leq n<2N}}\left( \sum _{m=1}^{k}\left(\beta (n+h_{m})+\chi _{{\cal P}}(n+h_{m})\right)-\rho \right)w_{n}.
\end{equation}
If $S^{'}(N, \rho)$ becomes positive for any sufficiently large $N$, there exists $n\in [N, 2N)$ such that at least $\rho +1$ of $n+h_{1}, \cdots ,n+h_{k}$ are in ${\cal P}\cup {\cal E}_{2}$. Hence our problem is to evaluate the sums
\begin{equation}
\label{4}
S_{0}=\underset{n\equiv \nu _{0} (\mathrm{mod}\; W)}{\sum _{N\leq n<2N}}w_{n}, \quad S_{1}^{(m)}=\underset{n\equiv \nu _{0} (\mathrm{mod} W)}{\sum _{N\leq n<2N}}\chi _{{\cal P}}(n+h_{m})w_{n},
\end{equation}
and 
\begin{equation}
\label{5}
S_{2}^{(m)}=\underset{n\equiv \nu _{0} (\mathrm{mod}\; W)}{\sum _{N\leq n<2N}}\beta (n+h_{m})w_{n}
\end{equation}
for $m=1, \cdots ,k$. Maynard (\cite{M1}, Proposition 4.1) computed the sums in (\ref{4}). The results are as follows:
\begin{prop}
We put 
\[
I_{k}(F)=\int _{0}^{1}\cdots \int _{0}^{1}F(t_{1}, \cdots ,t_{k})^{2}dt_{1}\cdots dt_{k},
\]
\[
J_{k}^{(m)}(F)=\int _{0}^{1}\cdots \int _{0}^{1}\left( \int _{0}^{1}F(t_{1},\cdots ,t_{k})dt_{m} \right)^{2}dt_{1}\cdots dt_{m-1}dt_{m+1}\cdots dt_{k}
\]
for $m=1,\cdots ,k$. Then, if $I_{k}(F)\neq 0$, we have 
\begin{equation}
\label{7}
S_{0}=\frac{(1+o(1))\varphi (W)^{k}N(\log R)^{k}}{W^{k+1}}I_{k}(F),
\end{equation}
and if $J_{k}^{(m)}(F)\neq 0$, we have
\begin{equation}
\label{7.5}
S_{1}^{(m)}=\frac{(1+o(1))\varphi (W)^{k}N(\log R)^{k+1}}{W^{k+1}\log N}J_{k}^{(m)}(F) \quad (m=1, \cdots ,k)
\end{equation}
as $N\to \infty$. 
\end{prop}
Hence the main problem of this paper is the computation of $S_{2}^{(m)}$. By substituting (\ref{2}) into (\ref{5}) and interchanging the summations, we have
\begin{equation}
\label{8}
S_{2}^{(m)}=\underset{e_{1},\cdots ,e_{k}}{\sum _{d_{1}, \cdots ,d_{k}}}\lambda _{d_{1}, \cdots ,d_{k}}\lambda _{e_{1}, \cdots ,e_{k}}\underset{[d_{i}, e_{i}]|n+h_{i} (\forall i)}{\underset{n\equiv \nu _{0} (\mathrm{mod}\; W)}{\sum _{N\leq n<2N}}}\beta (n+h_{m}).
\end{equation}
The integers $d_{m}$, $e_{m}$ must satisfy
\begin{equation}
\label{9}
d_{m}, e_{m} |n+h_{m}.
\end{equation}
Since $\beta (n+h_{m})=0$ unless $n+h_{m}=p_{1}p_{2}$, $Y<p_{1}\leq N^{\frac{1}{2}}<p_{2}$, and since $\lambda _{d_{1}, \cdots ,d_{k}}=0$ unless $\prod _{i=1}^{k}d_{i}< R<N^{\frac{1}{2}}$, only the following four types contribute to the sum above:\\
1) $d_{m}=p, e_{m}=1$ \quad $(Y<p< R)$, \\
2) $d_{m}=1, e_{m}=p$ \quad $(Y<p< R)$, \\
3) $d_{m}=e_{m}=1$, \\
4) $d_{m}=e_{m}=p$ \quad $(Y<p < R)$. \\
Correspondingly we decompose
\begin{equation}
\label{10}
S_{2}^{(m)}=S_{2,I}^{(m)}+S_{2, I\hspace{-.em}I}^{(m)}+S_{2, I\hspace{-.em}I\hspace{-.em}I}^{(m)}+S_{2, I\hspace{-.em}V}^{(m)}.
\end{equation}
The following three sections will be devoted to compute these terms. 

%%%%%%%%%%%%%%%%%%%%%%%%%%%%%%%%%%%%%%%%%%%%%%%%%%%%%%%%%%%%%%%%%%%%%%%%%%%%%%%%%%%%%%%%%%%%%%%%%%%%%%%%%%%%%%%%%%%%%%%%%%%%%%%%%%%%%%%%%%%%%%%%%%%%%%%%%%%%%%%%
\section{The computation of $S_{2,I}^{(m)}, S_{2, I\hspace{-.em}I}^{(m)}$}
We first compute $S_{2,I}^{(m)}(= S_{2, I\hspace{-.em}I}^{(m)})$. By interchanging the summations, we have
\begin{equation}
\label{11}
S_{2,I}^{(m)}=\sum _{Y<p < R}\underset{d_{m}=p, e_{m}=1}{\underset{e_{1},\cdots ,e_{k}}{\sum _{d_{1}, \cdots ,d_{k}}}}\lambda _{d_{1}, \cdots ,d_{k}}\lambda _{e_{1}, \cdots ,e_{k}}\underset{[d_{i}, e_{i}]|n+h_{i} (\forall i)}{\underset{n\equiv \nu _{0} (\mathrm{mod}\; W)}{\sum _{N\leq n<2N}}}\beta (n+h_{m}).
\end{equation}
By our choice of $\nu _{0}$ and the assumption that the elements of ${\cal H}$ are bounded, the inner sum is empty if $W, [d_{1},e_{1}], \cdots ,[d_{k},e_{k}]$ are not pairwise coprime.  We put 
\[
q=W\prod _{i=1}^{k}[d_{i}, e_{i}].
\]
When $W, [d_{1},e_{1}], \cdots ,[d_{k},e_{k}]$ are pairwise coprime, the sum over $n$ in (\ref{11}) is rewritten as a sum over a single residue class modulo $q$. That is, there exists a unique $\nu$ (mod $q$) such that $\nu \equiv \nu _{0}$ (mod $W$), $\nu +h_{i}\equiv 0$ (mod $[d_{i},e_{i}]$) and 
\begin{equation}
\label{12}
\underset{[d_{i}, e_{i}]|n+h_{i} (\forall i)}{\underset{n\equiv \nu _{0} (\mathrm{mod}\; W)}{\sum _{N\leq n<2N}}}\beta (n+h_{m})=\underset{ }{\underset{n\equiv \nu  (\mathrm{mod}\; q)}{\sum _{N\leq n<2N}}}\beta (n+h_{m})
\end{equation}
holds. We put $\nu _{m}=\nu +h_{m}$. Then, 
\begin{equation}
\label{p}
(\nu _{m}, q)=p.
\end{equation}
We will check this briefly. Since $[d_{m}, e_{m}]=p$, $p$ divides $q$, and since $\nu _{m}=\nu +h_{m}\equiv 0$ (mod $[d_{m}, e_{m}]$), $p$ divides $\nu _{m}$. Clearly, $p^{2}$ does not divide $q$. Let $p^{'}$ be another prime and assume that $p^{'}|(\nu _{m}, W)$. Since $p^{'}|\nu _{m}$, we have $\nu +h_{m}\equiv 0$ (mod $p^{'}$). Since $p^{'}|W$, we have  $\nu \equiv \nu _{0}$ (mod $p^{'}$). Therefore, we find that $\nu _{0} +h_{m}\equiv 0$ (mod $p^{'}$), hence $p^{'}|(\nu _{0} +h_{m}, W)$. This fact contradicts our assumption that $\nu _{0} +h_{m}$ is coprime to $W$. If $p^{'}|(\nu _{m}, d_{j})$ for some $j\neq m$, then $\nu +h_{m}\equiv 0$ (mod $p^{'}$) and $\nu \equiv -h_{j}$ (mod $p^{'}$). Hence $h_{m}\equiv h_{j}$ (mod $p^{'}$). However, since $d_{j}$ is coprime to $W$, the condition $p^{'}|d_{j}$ implies $p^{'}\geq \log \log \log N$. Therefore, the conclusion that $h_{m}\equiv h_{j}$ (mod $p^{'}$) ($j\neq m$) contradicts our assumption that the elements of ${\cal H}$ are bounded. Hence $p^{'}$ does not divide $(\nu _{m}, d_{j})$. In a similar way, we find that $p^{'}$ does not divide $(\nu _{m}, e_{j})$. Thus we obtain (\ref{p}). 

Therefore, there exists a unique $\nu _{m}^{'}$ (mod $q/p$) such that $p\nu _{m}^{'}\equiv \nu _{m}$ (mod $q$) and the right hand side of (\ref{12}) becomes  
\begin{equation}
\label{13}
\begin{aligned}
\underset{ }{\underset{n\equiv \nu  (\mathrm{mod}\; q)}{\sum _{N\leq n<2N}}}\beta (n+h_{m})&=\underset{ }{\underset{n\equiv \nu _{m}  (\mathrm{mod}\; q)}{\sum _{N\leq n<2N}}}\beta (n)+O(1) \\
&= \underset{ }{\underset{n^{'}\equiv \nu _{m}^{'}  (\mathrm{mod}\; \frac{q}{p})}{\sum _{\frac{N}{p}\leq n^{'}<\frac{2N}{p}}}}\beta (pn^{'})+O(1) \\
&= \underset{ }{\underset{n^{'}\equiv \nu _{m}^{'}  (\mathrm{mod}\; \frac{q}{p})}{\sum _{\frac{N}{p}\leq n^{'}<\frac{2N}{p}}}}\chi _{{\cal P}}(n^{'})+O(1).
\end{aligned}
\end{equation}
We note that $(\nu _{m}^{'}, q/p)=1$, by (\ref{p}). The sum in the right hand side of (\ref{13}) becomes 
\begin{equation}
\label{14}
\begin{aligned}
\underset{ }{\underset{n^{'}\equiv \nu _{m}^{'}  (\mathrm{mod}\; \frac{q}{p})}{\sum _{\frac{N}{p}\leq n^{'}<\frac{2N}{p}}}}\chi _{{\cal P}}(n^{'})&=\frac{1}{\varphi \left(\frac{q}{p}\right)}\sum _{\frac{N}{p}\leq n^{'}<\frac{2N}{p}}\chi _{{\cal P}}(n^{'})+\Delta \left( \frac{N}{p}; \frac{q}{p}, \nu _{m}^{'} \right) \\
&=\frac{1}{\varphi \left(\frac{q}{p}\right)}\pi ^{\flat}\left( \frac{N}{p}\right) +\Delta \left( \frac{N}{p}; \frac{q}{p}, \nu _{m}^{'} \right),
\end{aligned}
\end{equation} 
where 
\[
\Delta (N; q,a)=\underset{n\equiv a (\mathrm{mod}\; q)}{\sum _{N \leq n<2N}}\chi _{{\cal P}}(n)-\frac{\pi ^{\flat}(N)}{\varphi (q)}.
\]
By combining (\ref{13}), (\ref{14}), we have
\begin{equation}
\label{15}
\underset{[d_{i}, e_{i}]|n+h_{i} (\forall i)}{\underset{n\equiv \nu _{0} (\mathrm{mod}\; W)}{\sum _{N\leq n<2N}}}\beta (n+h_{m})=\frac{\pi ^{\flat} \left( \frac{N}{p} \right) }{\varphi \left( \frac{q}{p} \right) } +\Delta \left( \frac{N}{p}; \frac{q}{p}, \nu _{m}^{'}\right)+O(1).
\end{equation}
By substituting  (\ref{15}) into (\ref{11}), we obtain
\begin{equation}
\label{16}
\begin{aligned}
S_{2,I}^{(m)}=& \frac{1}{\varphi (W)}\sum _{Y<p < R}\pi ^{\flat} \left( \frac{N}{p} \right)  \underset{d_{m}=p, e_{m}=1}{\underset{e_{1}, \cdots ,e_{k}}{\sum _{d_{1}, \cdots ,d_{k}}^{\quad\quad '}}}\frac{\lambda _{d_{1}, \cdots ,d_{k}}\lambda _{e_{1},\cdots ,e_{k}}}{\prod _{i\neq m}\varphi ([d_{i}, e_{i}])} \\
&\quad+O\left( \sum _{Y<p< R} \underset{d_{m}=p, e_{m}=1}{\underset{e_{1}, \cdots ,e_{k}}{\sum ^{\quad \;'} _{d_{1}, \cdots ,d_{k}}}}|\lambda _{d_{1}, \cdots ,d_{k}}\lambda _{e_{1}, \cdots ,e_{k}}| \left( \left| \Delta \left( \frac{N}{p}; \frac{q}{p}, \nu _{m}^{'} \right) \right| +1 \right)   \right), 
\end{aligned}
\end{equation}
where the sum $\sum ^{ '}$ indicates that $d_{1}, \cdots ,d_{k}, e_{1}, \cdots ,e_{k}$ are restricted to those satisfying the condition that $W, [d_{1}, e_{1}], \cdots ,[d_{k}, e_{k}]$ are pairwise coprime.
We now evaluate the  error term  of (\ref{16}). The conductor $q$ is square-free, and satisfies $q< R^{2}W$. Moreover, $p$ divides $q$. The number of pairs $(d_{1}, \cdots ,d_{k}, e_{1}, \cdots , e_{k})$  satisfying 
\[
q=W\prod _{i=1}^{k}[d_{i}, e_{i}]
\]
is at most $\tau _{3k}(q)$. Therefore, by the Cauchy-Schwarz inequality, we have
\begin{align*}
&\sum _{Y<p< R} \underset{d_{m}=p, e_{m}=1}{\underset{e_{1}, \cdots ,e_{k}}{\sum ^{\quad \; '} _{d_{1}, \cdots ,d_{k}}}}|\lambda _{d_{1}, \cdots ,d_{k}}\lambda _{e_{1}, \cdots ,e_{k}}| \left( \left| \Delta \left( \frac{N}{p}; \frac{q}{p}, \nu _{m}^{'} \right) \right| +1 \right) \\
& \quad \ll \lambda _{\mathrm{max}}^{2}\sum _{Y<p< R}\underset{p|q}{\sum _{q<R^{2}W}}\mu ^{2}(q) \tau _{3k}(q)\left(\left| \Delta \left( \frac{N}{p}; \frac{q}{p}, \nu _{m}^{'} \right) \right| +1 \right) \\
& \quad \ll \lambda _{\mathrm{max}}^{2}  \sum _{Y<p< R}\left( \sum _{q^{'}<\frac{R^{2}W}{p}}\mu ^{2}(pq^{'})\tau _{3k}(pq^{'})^{2}\frac{N}{p\varphi (q^{'})} \right)^{\frac{1}{2}}\left(\sum _{q^{'}<\frac{R^{2}W}{p}}\mu ^{2}(q^{'})\Delta ^{*}\left( \frac{N}{p}; q^{'} \right) \right)^{\frac{1}{2}}, 
\end{align*}
where 
\[
\lambda _{\mathrm{max}}=\sup _{d_{1}, \cdots ,d_{k}}|\lambda _{d_{1},  \cdots ,d_{k}}|, \quad
 \Delta ^{*}(N;q)=\max _{(a,q)=1}|\Delta (N; q,a)|.
\]
Under the assumption of $BV[\theta ,{\cal P}]$, the above is at most 
\begin{equation}
\label{17}
\begin{aligned}
&\ll _{A}\lambda _{\mathrm{max}}^{2}\sum _{Y<p< R}p^{-\frac{1}{2}}N^{\frac{1}{2}}(\log N)^{a_{k}}\cdot \frac{\left( \frac{N}{p} \right)^{\frac{1}{2}}}{(\log N)^{A} } \\
& \ll _{A} \lambda _{\mathrm{max}}^{2}N(\log N)^{a_{k}-A}\sum _{Y<p < R}\frac{1}{p}  \ll _{B}\frac{\lambda _{\mathrm{max}}^{2}N}{(\log N)^{B}},
\end{aligned}
\end{equation}
where $a_{k}$ is some positive integer depending only on $k$, and $A$ is an arbitrary positive number, and $B=A-a_{k}-1$. Here, we evaluated the first $q^{'}$-sum by a standard method.  Hence we may regard  $B$ as an arbitrary positive number, once $k$ is fixed. Combining (\ref{16}), (\ref{17}), we have
\begin{equation}
\label{18}
\begin{aligned}
S_{2, I}^{(m)}=&\frac{1}{\varphi (W)}\sum _{Y<p < R}(p-1) \pi ^{\flat}\left( \frac{N}{p} \right) \underset{d_{m}=p, e_{m}=1}{\underset{e_{1}, \cdots ,e_{k}}{\sum _{d_{1}, \cdots ,d_{k}}^{\quad\quad '}}}\frac{\lambda _{d_{1}, \cdots ,d_{k}}\lambda _{e_{1}, \cdots ,e_{k}}}{\prod _{i=1}^{k}\varphi ([d_{i}, e_{i}])} \\
& +O_{B}\left( \frac{\lambda _{\mathrm{max}}^{2}N}{(\log N)^{B}} \right).
\end{aligned}
\end{equation}
We used the fact that
\[
\frac{1}{\prod _{i\neq m}\varphi ([d_{i},e_{i}])}=\frac{p-1}{\prod _{i=1}^{k}\varphi ([d_{i},e_{i}])}
\]
for $d_{m}=p, e_{m}=1$. Next we compute the sum 
\[
\underset{d_{m}=p, e_{m}=1}{\underset{e_{1}, \cdots ,e_{k}}{\sum _{d_{1}, \cdots ,d_{k}}^{\quad\quad '}}}\frac{\lambda _{d_{1}, \cdots ,d_{k}}\lambda _{e_{1}, \cdots ,e_{k}}}{\prod _{i=1}^{k}\varphi ([d_{i}, e_{i}])} .
\]
Let $g$ be the totally multiplicative function defined by $g(q)=q-2$ for $q\in {\cal P}$. Then, when $d_{i}, e_{i}$ are square-free, we have
\[
\frac{1}{\varphi ([d_{i}, e_{i}])}=\frac{1}{\varphi (d_{i})\varphi (e_{i})}\sum _{u_{i}|d_{i}, e_{i}}g(u_{i}).
\]
Moreover, the condition that $(d_{i}, e_{j})=1$ ($i\neq j$) is replaced by multiplying $\sum _{s_{i,j}|d_{i}, e_{j}}\mu (s_{i,j})$. Since $\lambda _{d_{1}, \cdots ,d_{k}}=0$ unless $(d_{i}, d_{j})=1$ ($\forall i\neq j$), we may add the condition that  $s_{i,j}$ is coprime to $u_{i}, u_{j}, s_{i,a} (a\neq j), s_{b,j} (b\neq i)$. We denote by $\sum ^{*}$ the sum over $s_{1,2}, \cdots ,s_{k,k-1}$ restricted to those satisfying this condition. Then we have
\begin{equation}
\label{19}
\begin{aligned}
\underset{d_{m}=p, e_{m}=1}{\underset{e_{1}, \cdots ,e_{k}}{\sum _{d_{1}, \cdots ,d_{k}}^{\quad\quad '}}}\frac{\lambda _{d_{1}, \cdots ,d_{k}}\lambda _{e_{1}, \cdots ,e_{k}}}{\prod _{i=1}^{k}\varphi ([d_{i}, e_{i}])}=& \sum _{u_{1},\cdots ,u_{k}}\prod _{i=1}^{k}g(u_{i})\sum _{s_{1,2}, \cdots ,s_{k,k-1}}^{\quad\quad *}\left( \prod _{1\leq i\neq j\leq k}\mu (s_{i,j}) \right) \\
&\quad \quad \quad \quad \quad  \times \underset{d_{m}=p, e_{m}=1}{\underset{s_{i,j}|d_{i},e_{j} (i\neq j)}{\underset{u_{i}|d_{i}, e_{i} (\forall i)}{\underset{e_{1}, \cdots ,e_{k}}{\sum _{d_{1}, \cdots ,d_{k}}}}}}   \frac{\lambda _{d_{1}, \cdots ,d_{k}}\lambda _{e_{1}, \cdots ,e_{k}}}{\prod _{i=1}^{k}\varphi (d_{i})\varphi (e_{i})}.
\end{aligned}
\end{equation}
We put
\begin{equation}
\label{20}
y_{r_{1}, \cdots ,r_{k}}^{(m)}(p)=\left( \prod _{i=1}^{k}\mu (r_{i})g(r_{i})\right) \underset{d_{m}=p}{\underset{r_{i}|d_{i} (\forall i)}{\sum _{d_{1}, \cdots ,d_{k}}}}\frac{\lambda _{d_{1} ,\cdots ,d_{k}}}{\prod _{i=1}^{k}\varphi (d_{i})},
\end{equation}
\begin{equation}
\label{21}
y_{r_{1}, \cdots ,r_{k}}^{(m)}=\left( \prod _{i=1}^{k}\mu (r_{i})g(r_{i})\right) \underset{d_{m}=1}{\underset{r_{i}|d_{i} (\forall i)}{\sum _{d_{1}, \cdots ,d_{k}}}}\frac{\lambda _{d_{1} ,\cdots ,d_{k}}}{\prod _{i=1}^{k}\varphi (d_{i})},
\end{equation}
and $r:=\prod _{i=1}^{k}r_{i}$. Then, $y_{r_{1}, \cdots ,r_{k}}^{(m)}(p)=0$ unless $r$ is square-free, $(r,W)=1$, $r<R$, and $r_{m}=1$ or $p$. Similarly, $y_{r_{1}, \cdots ,r_{k}}^{(m)}=0$ unless $r$ is square-free, $(r,W)=1$, $r<R$, and $r_{m}=1$. Then the right hand side of (\ref{19}) is  expressed by 
\begin{equation}
\label{22}
\sum _{u_{1} ,\cdots ,u_{k}}\left( \prod _{i=1}^{k}\frac{\mu ^2(u_{i})}{g(u_{i})} \right)\sum _{s_{1,2}, \cdots ,s_{k,k-1}}^{\quad\quad *}\left( \prod _{1\leq i\neq j\leq k}\frac{\mu (s_{i,j})}{g(s_{i,j})^{2}} \right)y_{a_{1},\cdots ,a_{k}}^{(m)}(p)y_{b_{1}, \cdots ,b_{k}}^{(m)},
\end{equation}
where $a_{i}=u_{i}\prod _{j\neq i}s_{i,j}$, $b_{i}=u_{i}\prod _{j\neq i}s_{j,i}$. To obtain (\ref{22}), we used $\mu (a_{i})=\mu (u_{i})\prod _{j\neq i}\mu (s_{i,j})$, $\mu (b_{i})=\mu (u_{i})\prod _{j\neq i}\mu (s_{j,i})$, $g(a_{i})=g(u_{i})\prod _{j\neq i}g (s_{i,j})$ and $g(b_{i})=g(u_{i})\prod _{j\neq i}g(s_{j,i})$. Since $s_{i,j}$ is coprime to $u_{i}, u_{j}, s_{i,a} (a\neq j), s_{b,j} (b\neq i)$, these identities hold. We consider the contribution of the terms with $s_{i,j}\neq 1$ to (\ref{22}). By the condition of the support of $y_{r_{1},\cdots ,r_{k}}^{(m)}$, only the terms with $s_{i,j}=1$ or $s_{i,j}>D_{0}$ contribute to the sum above. Hence the contribution of the terms with $s_{i,j}\neq 1, a_{m}=1$ is at most 
\begin{equation}
\label{23}
\begin{aligned}
&y_{\mathrm{max}}^{(m)}(p)|_{r_{m}=1}\;y_{\mathrm{max}}^{(m)}\left( \sum _{\underset{(u,W)=1}{u<R}}\frac{\mu ^{2}(u)}{g(u)} \right)^{k-1}\left( \sum _{s} \frac{\mu ^{2}(s)}{g(s)^{2}} \right)^{k^{2}-k-1}\left( \sum _{s_{i,j}>D_{0}}\frac{\mu ^{2}(s_{i,j})}{g(s_{i,j})^{2}} \right) \\
&\quad \ll y_{\mathrm{max}}^{(m)}(p)|_{r_{m}=1}\;y_{\mathrm{max}}^{(m)}\left( \frac{\varphi (W)}{W}\log R \right)^{k-1}\cdot 1\cdot D_{0}^{-1} \\
&\quad \ll \frac{\varphi (W)^{k-1}(\log R)^{k-1}}{W^{k-1}D_{0}}y_{\mathrm{max}}^{(m)}(p)|_{r_{m}=1}\;y_{\mathrm{max}}^{(m)},
\end{aligned}
\end{equation}
where
\[
y_{\mathrm{max}}^{(m)}(p)|_{r_{m}=\mathfrak{p}}:=\underset{r_{m}=\mathfrak{p}}{\sup _{r_{1}, \cdots ,r_{k}}}| y_{r_{1}, \cdots ,r_{k}}^{(m)}(p)| \quad (\mathfrak{p}=1\; \mathrm{or} \; p), \quad 
y_{\mathrm{max}}^{(m)}:=\sup _{r_{1}, \cdots ,r_{k}}|y_{r_{1}, \cdots ,r_{k}}^{(m)}|.
\]
Similarly, the contribution of the terms with $s_{i,j}\neq 1$, $a_{m}=p$ is, since in this case $u_{m}$ or some $s_{m,j}$ is equal to $p$, at most
\begin{equation}
\label{24}
\frac{\varphi (W)^{k-1}(\log R)^{k-1}}{pW^{k-1}D_{0}}y_{\mathrm{max}}^{(m)}(p)|_{r_{m}=p}\;y_{\mathrm{max}}^{(m)}.
\end{equation}
Combining (\ref{19}), (\ref{22}),  (\ref{23}) and (\ref{24}), we have
\begin{equation}
\label{25}
\begin{aligned}
&\underset{d_{m}=p, e_{m}=1}{\underset{e_{1}, \cdots ,e_{k}}{\sum _{d_{1}, \cdots ,d_{k}}^{\quad\quad '}}}\frac{\lambda _{d_{1}, \cdots ,d_{k}}\lambda _{e_{1}, \cdots ,e_{k}}}{\prod _{i=1}^{k}\varphi ([d_{i}, e_{i}])}\\
&\quad = \sum _{u_{1}, \cdots ,u_{k}}\frac{y_{u_{1}, \cdots ,u_{k}}^{(m)}(p)y_{u_{1}, \cdots ,u_{k}}^{(m)}}{\prod _{i=1}^{k}g(u_{i})} \\
&\quad\quad +O\left( \frac{\varphi (W)^{k-1}(\log R)^{k-1}}{W^{k-1}D_{0}}\left( y_{\mathrm{max}}^{(m)}(p)|_{r_{m}=1}+\frac{y_{\mathrm{max}}^{(m)}(p)|_{r_{m}=p}}{p}\right) y_{\mathrm{max}}^{(m)} \right).
\end{aligned}
\end{equation}
(We note that we may remove $\mu ^{2} (u_{i})$ in (\ref{22}), since $y_{u_{1},\cdots ,u_{k}}^{(m)}=0$ unless $u_{1},\cdots ,u_{k}$ are all square-free.) We put 
\begin{equation}
\label{y}
y_{r_{1},\cdots ,r_{k}}=\left( \prod _{i=1}^{k}\mu (r_{i})g(r_{i})\right) \underset{r_{i}|d_{i} (\forall i)}{\sum _{d_{1}, \cdots ,d_{k}}}\frac{\lambda _{d_{1} ,\cdots ,d_{k}}}{\prod _{i=1}^{k}d_{i}}.
\end{equation}
Using the test function $F$, this is expressed by 
\begin{equation}
\label{yF}
y_{r_{1}, \cdots ,r_{k}}=F\left( \frac{\log r_{1}}{\log R}, \cdots ,\frac{\log r_{k}}{\log R} \right)
\end{equation}
(see \cite{M1}, p.400). It is proved in \cite{M1} that 
\[
\lambda _{\mathrm{max}}\ll y_{\mathrm{max}}(\log R)^{k},
\]
where 
\[
y_{\mathrm{max}}=\sup _{r_{1}, \cdots ,r_{k}}|y_{r_{1}, \cdots ,r_{k}}|.
\]
Hence the error term in (\ref{18}) is replaced by $y_{\mathrm{max}}^{2}N/(\log N)^{B}$. We substitute (\ref{25}) into (\ref{18}). Since 
\[
(p-1) \pi ^{\flat}\left( \frac{N}{p} \right)=\frac{N}{\log \frac{N}{p}}+O _{\eta}\left(  \frac{N}{(\log N )^{2}} \right)
\]
for $Y=N^{\eta}<p < R=N^{\frac{\theta}{2}-\delta}$,  we obtain the following result: 
\begin{lem}
Assume $BV[\theta ,{\cal P}]$ for $0<\theta \leq 1$. Then
\begin{equation}
\label{26}
\begin{aligned}
S_{2,I}^{(m)}=&\frac{N}{\varphi (W)}  \left( 1+O\left( \frac{1}{\log N} \right) \right)   \sum _{Y<p< R}\frac{1}{\log \frac{N}{p}}\sum _{u_{1}, \cdots ,u_{k}}\frac{y_{u_{1}, \cdots ,u_{k}}^{(m)}(p)y_{u_{1}, \cdots ,u_{k}}^{(m)}}{\prod _{i=1}^{k}g(u_{i})} \\
&+O\left( \frac{N\varphi (W)^{k-2}(\log N)^{k-2}}{W^{k-1}D_{0}}y_{\mathrm{max}}^{(m)}\sum _{Y<p< R}\left( y_{\mathrm{max}}^{(m)}(p)|_{r_{m}=1}+\frac{ y_{\mathrm{max}}^{(m)}(p)|_{r_{m}=p}}{p} \right) \right) \\
&+O_{B}\left( \frac{Ny_{\mathrm{max}}^{2}}{(\log N)^{B}}\right).
\end{aligned}
\end{equation}
\end{lem}
By symmetry, the same result also holds for $S_{2, I\hspace{-.em}I}^{(m)}$. Next, we compute the inner sum in the main term of (\ref{26}). The following result is obtained by Maynard (\cite{M1}, Lemma 5.3).
\begin{lem}
If $r_{m}=1$, we have
\begin{equation}
\label{27}
y_{r_{1},\cdots ,r_{k}}^{(m)}=\sum _{a_{m}}\frac{y_{r_{1}, \cdots, r_{m-1}, a_{m}, r_{m+1}, \cdots ,r_{k}}}{\varphi (a_{m})}+O\left( \frac{y_{\mathrm{max}}\varphi (W)\log R}{WD_{0}} \right).
\end{equation}
\end{lem}
Next, if $\prod _{i=1}^{k}d_{i}$ is square-free, we have
\[
\lambda _{d_{1}, \cdots ,d_{k}}=\left( \prod _{i=1}^{k}\mu (d_{i})d_{i} \right)\underset{d_{i}|a_{i} (\forall i)}{\sum _{a_{1}, \cdots ,a_{k}}}\frac{y_{a_{1},\cdots ,a_{k}}}{\prod _{i=1}^{k}\varphi (a_{i})}
\]
(see \cite{M1}, p.393, (5.8)). By substituting  this into (\ref{20}) and interchanging the order of summation, we have
\[
y_{r_{1}, \cdots ,r_{k}}^{(m)}(p)=\left( \prod _{i=1}^{k}\mu (r_{i})g(r_{i}) \right) \underset{p|a_{m}}{\underset{r_{i}|a_{i} (\forall i)}{\sum _{a_{1}, \cdots ,a_{k}}}}\frac{y_{a_{1},\cdots ,a_{k}}}{\prod _{i=1}^{k}\varphi (a_{i})}  \underset{r_{i}|d_{i}, d_{i}|a_{i} (\forall i)} {\underset{d_{m}=p} {\sum _{d_{1},\cdots ,d_{k}}}}\frac{\prod _{i=1}^{k}\mu (d_{i})d_{i}}{\prod _{i=1}^{k}\varphi (d_{i})}.
\]
If $r_{m}=1$ or $p$, we find that
\begin{align*}
\underset{r_{i}|d_{i}, d_{i}|a_{i} (\forall i)} {\underset{d_{m}=p} {\sum _{d_{1},\cdots ,d_{k}}}}\frac{\prod _{i=1}^{k}\mu (d_{i})d_{i}}{\prod _{i=1}^{k}\varphi (d_{i})} &=\frac{\mu (p)p}{\varphi (p)}\underset{r_{i}|d_{i}, d_{i}|a_{i} (\forall i \neq m)}{\sum _{d_{1}, \cdots ,d_{m-1},d_{m+1}, \cdots ,d_{k}}}\frac{\prod _{i\neq m}\mu (d_{i})d_{i}}{\prod _{i\neq m}\varphi (d_{i})} \\
&=-\frac{p}{p-1}\prod _{i\neq m}\frac{\mu (a_{i})r_{i}}{\varphi (a_{i})}.
\end{align*}
Therefore, 
\begin{equation}
\label{28}
y_{r_{1}, \cdots, r_{k}}^{(m)}(p)=-\frac{p}{p-1}\left( \prod_{i=1}^{k}\mu (r_{i})g(r_{i})\right) \underset{p|a_{m}}{\underset{r_{i}|a_{i} (\forall i)}{\sum _{a_{1},\cdots ,a_{k}}}}\frac{y_{a_{1},\cdots ,a_{k}}}{\prod _{i=1}^{k}\varphi (a_{i})}\prod _{i\neq m}\frac{\mu (a_{i})r_{i}}{\varphi (a_{i})}.
\end{equation}
By the condition of the support of $y_{a_{1}, \cdots ,a_{k}}$, we may restrict the sum to $(a_{i},W)=1$ ($\forall i$). Then, if $a_{j}\neq r_{j}$, it follows that $a_{j}>D_{0}r_{j}$.
 For $j\neq m$, the contribution of such terms is at most
\begin{align*}
&y_{\mathrm{max}}r_{m}^{-1}\left( \prod _{i=1}^{k}g(r_{i})r_{i}\right)\left( \underset{r_{j}|a_{j}}{\sum _{a_{j}>D_{0}r_{j}}}\frac{\mu ^{2}(a_{j})}{\varphi (a_{j})^{2}}\right) \prod _{i\neq j,m}\left( \sum _{r_{j}|a_{j}}\frac{\mu ^{2}(a_{j})}{\varphi (a_{j})^{2}}\right) \underset{(a_{m},W)=1}{\underset{a_{m}<R}{\sum _{p|a_{m}}}}\frac{\mu ^{2}(a_{m})}{\varphi (a_{m})} \\
&\ll y_{\mathrm{max}}\left( \prod_{i\neq m}\frac{g(r_{i})r_{i}}{\varphi (r_{i})^{2}}\right)\cdot g(r_{m})\cdot D_{0}^{-1}\cdot 1\cdot \underset{(a_{m}^{'},W)=1}{\sum_{a_{m}^{'}<\frac{R}{p}}}\frac{\mu ^{2}(a_{m}^{'})}{\varphi (p)\varphi (a_{m}^{'})} \\
&\ll \frac{y_{\mathrm{max}}g(r_{m})\varphi (W)\log \frac{R}{p}}{WD_{0}\varphi (p)}.
\end{align*}
Hence we have
\begin{equation}
\label{29}
\begin{aligned}
y_{r_{1},\cdots ,r_{k}}^{(m)}(p)=&-\frac{p}{p-1}\left( \prod _{i\neq m}\frac{\mu ^{2}(r_{i})g(r_{i})r_{i}}{\varphi (r_{i})^{2}} \right)\mu (r_{m})g(r_{m})\sum _{p|a_{m}}\frac{y_{r_{1}, \cdots ,r_{m-1},a_{m}, r_{m+1}, \cdots ,r_{k}}}{\varphi (a_{m})} \\
& \quad  +O\left( \frac{y_{\mathrm{max}}g(r_{m})\varphi (W)\log \frac{R}{p}}{WD_{0}\varphi (p)} \right).
\end{aligned}
\end{equation}
Since $y_{r_{1},\cdots ,r_{k}}=0$ unless $r_{1}, \cdots ,r_{k}$ are square-free,  we may remove the factors $\mu ^{2}(r_{i})$ ($i\neq m$). Finally, by applying 
\[
\frac{p}{p-1}=1+O(N^{-\eta}),\quad \prod _{i\neq m}\frac{g(r_{i})r_{i}}{\varphi (r_{i})^{2}}=1+O(D_{0}^{-1}) \quad (\mathrm{if}\; (\prod _{i\neq m}r_{i},W)=1),
\]
we obtain the following result: 
\begin{lem}
If $r_{m}=1\;  \mathrm{or}\;  p$, we have 
\begin{equation}
\label{30}
\begin{aligned}
y_{r_{1}, \cdots ,r_{k}}^{(m)}(p)=&-\mu (r_{m})g(r_{m})\underset{p|a_{m}}{\sum _{a_{m}}}\frac{y_{r_{1},\cdots ,r_{m-1}, a_{m}, r_{m+1}, \cdots ,r_{k}}}{\varphi (a_{m})} \\
&\quad + O\left( \frac{y_{\mathrm{max}}g(r_{m})\varphi (W)\log \frac{R}{p}}{WD_{0}\varphi (p)} \right).
\end{aligned}
\end{equation}
\end{lem}
By (\ref{30}), we have
\begin{equation}
\label{31}
\begin{aligned}
y_{\mathrm{max}}^{(m)}(p)|_{r_{m}=p}\ll \frac{y_{\mathrm{max}}\varphi (W)\log \frac{R}{p}}{W}, 
\end{aligned}
\end{equation}
and 
\begin{equation}
\label{32}
y_{\mathrm{max}}^{(m)}(p)|_{r_{m}=1}\ll \frac{y_{\mathrm{max}}\varphi (W)\log \frac{R}{p}}{pW}.
\end{equation}
Next we compute the sum over $a_{m}$. For this purpose, we use the following lemma, proved in \cite{GGPY2} (see Lemma 6.2 of \cite{M1}).
\begin{lem}
Let $A_{1}, A_{2}, L>0$ and $\gamma$ be a multiplicative function satisfying 
\[
0\leq \frac{\gamma (q)}{q} \leq 1-A_{1},
\]
\[
-L\leq \underset{q\in {\cal P}}{\sum _{w\leq q\leq z}}\frac{\gamma (q)\log q}{q}-\log \frac{z}{w}\leq A_{2}
\]
for any $2\leq w\leq z$. Let $h$ be the totally multiplicative function defined by
\[
h(q)=\frac{\gamma (q)}{q-\gamma (q)}
\]
for primes $q$. For a smooth function $G: [0,1]\to \mathbf{R}$, put 
\[
G_{\mathrm{max}}=\sup _{t\in [0,1]} (|G(t)|+|G^{'}(t)|).
\]
Then, we have
\[
\sum _{d<z}\mu ^{2}(d)h(d)G\left( \frac{\log d}{\log z}\right)=\mathfrak{S} \log z \int _{0}^{1}G(x)dx +O(\mathfrak{S} LG_{\mathrm{max}}),
\]
where 
\[
\mathfrak{S} =\prod _{q\in {\cal P}}\left( 1-\frac{\gamma (q)}{q} \right)^{-1}\left(1-\frac{1}{q}\right).
\]
\end{lem}
The following lemma is a direct consequence of Lemma 3.4:
\begin{lem}
Under the same situation as in Lemma $3.4$, put
\[
G_{p}(x)=G\left( \frac{\log \frac{R}{p}}{\log R}\left( \frac{\log p}{\log \frac{R}{p}}+x \right) \right).
\]
Then, we have 
\[
\sum _{d<\frac{R}{p}}\mu ^{2}(d)h(d)G\left( \frac{\log pd}{\log R} \right)=\mathfrak{S} \log \frac{R}{p}\int _{0}^{1}G_{p}(x)dx +O(\mathfrak{S} LG_{\mathrm{max}}).
\]
\end{lem}
\begin{proof}
Since 
\[
G\left( \frac{\log pd}{\log R}\right) =G_{p}\left( \frac{\log d}{\log \frac{R}{p}}\right), \quad (G_{p})_{\mathrm{max}}\ll G_{\mathrm{max}},
\]
by applying Lemma 3.4 with $z=R/p$, we obtain the result. 
\end{proof}
We compute the sum in (\ref{30}). Using the conditions of the support of $y_{r_{1}, \cdots ,r_{k}}$, we have
\begin{equation}
\label{33}
\begin{aligned}
&\underset{p|a_{m}}{\sum _{a_{m}}}\frac{y_{r_{1},\cdots ,r_{m-1},a_{m},r_{m+1}, \cdots ,r_{k}}}{\varphi (a_{m})} \\
& =\frac{1}{\varphi (p)}\underset{(p,a_{m}^{'})=1}{\sum _{a_{m}^{'}}}\frac{y_{r_{1}, \cdots ,r_{m-1},pa_{m}^{'}, r_{m+1}, \cdots, r_{k}}}{\varphi (a_{m}^{'})} \\
& =\frac{1}{\varphi (p)} \underset{(a_{m}^{'}, pW\prod _{i\neq m}r_{i})=1}{\sum _{a_{m}^{'}<\frac{R}{p}}}\frac{\mu ^{2}(a_{m}^{'})}{\varphi (a_{m}^{'})}F\left( \frac{\log r_{1}}{\log R}, \cdots ,\frac{\log r_{m-1}}{\log R}, \frac{\log pa_{m}^{'}}{\log R}, \frac{\log r_{m+1}}{\log R}, \cdots ,\frac{\log r_{k}}{\log R} \right).
\end{aligned}
\end{equation}
We apply Lemma 3.5 with 
\begin{eqnarray}
\gamma (q)=\left\{ \begin{array}{ll}
 1 & (q\mid \hspace{-.67em}/ pW\prod _{i\neq m}r_{i}) \\
 0 & (\mathrm{otherwise}). \\
\end{array} \right.
\end{eqnarray} 
In this case, we have
\begin{align*}
L\ll &1+\sum _{q|pW\prod _{i\neq m}r_{i}}\frac{\log q}{q} \\
\ll & \sum _{q<\log R}\frac{\log q}{q} +\underset{q \geq \log R}{\sum _{q|pW\prod _{i\neq m}r_{i}}}\frac{\log \log R}{\log R} \\
\ll &\log \log N.
\end{align*}
Moreover, since $(r_{i}, p)=1$ ($\forall i\neq m$), we have
\[
\mathfrak{S} =\prod _{q|pW\prod _{i\neq m}r_{i}}\left( 1-\frac{1}{q}\right) =\frac{\varphi (p)\varphi (W)}{pW}\prod _{i\neq m}\frac{\varphi (r_{i})}{r_{i}}.
\]
Therefore, by applying Lemma 3.5 to the right hand side of (\ref{33}) , we obtain
\begin{equation}
\label{34}
\begin{aligned}
& \underset{p|a_{m}}{\sum _{a_{m}}}\frac{y_{r_{1},\cdots ,r_{m-1},a_{m},r_{m+1}, \cdots ,r_{k}}}{\varphi (a_{m})} \\
&\quad =\frac{\varphi (W)}{pW}\log \frac{R}{p}\prod _{i\neq m}\frac{\varphi (r_{i})}{r_{i}}\int _{0}^{1}F_{p}^{[m]}\left( \frac{\log r_{1}}{\log R}, \cdots ,\frac{\log r_{m-1}}{\log R}, u, \frac{\log r_{m+1}}{\log R}, \cdots ,\frac{\log r_{k}}{\log R} \right) du \\
&\quad \quad +O\left( \frac{\varphi (W)}{pW}F_{\mathrm{max}}\log \log N \right),
\end{aligned}
\end{equation}
where the function $F_{p}^{[m]}(\cdots )$ is obtained by replacing the $m$-th component $x$ of $F(\cdots)$ with $((\log R/p)/\log R)((\log p/\log R/p) +x)$. We put
\begin{equation}
\label{35}
F_{r_{1},\cdots ,r_{k}}^{(p;m)}=\int _{0}^{1}F_{p}^{[m]}\left( \frac{\log r_{1}}{\log R}, \cdots ,\frac{\log r_{m-1}}{\log R}, u, \frac{\log r_{m+1}}{\log R}, \cdots ,\frac{\log r_{k}}{\log R} \right) du.
\end{equation}
By substituting  (\ref{34}) into (\ref{30}) and using $y_{\mathrm{max}}\ll F_{\mathrm{\max}}$, $\log R/p \ll \log R$,  we obtain the following result: 
\begin{lem}
If $r_{m}=1$, we have
\begin{equation}
\label{36}
y_{r_{1},\cdots ,r_{k}}^{(m)}(p)=-\frac{\varphi (W)}{pW}\log \frac{R}{p} \left(  \prod _{i\neq m}\frac{\varphi (r_{i})}{r_{i}} \right)F_{r_{1}, \cdots ,r_{k}}^{(p;m)}+O\left( \frac{F_{\mathrm{max}}\varphi (W)\log R}{pWD_{0}} \right)
\end{equation}
for $Y<p< R$, where $F_{r_{1}, \cdots ,r_{k}}^{(p;m)}$ is defined by (\ref{35}). 
\end{lem}
The summand in the main term of (\ref{26}) is zero unless $u_{1}, \cdots ,u_{k}$ satisfy $u_{m}=1$, $(u_{i},u_{j})=1$ ($i\neq j$), $\prod _{i=1}^{k}u_{i}$ is square-free and $(u_{i}, pW)=1$ ($\forall i$). It is proved in \cite{M1} (p.403, (6.13)) that if $u_{1}, \cdots ,u_{k}$ satisfy these conditions, we have
\begin{equation}
\label{37}
y_{u_{1}, \cdots ,u_{k}}^{(m)}=(\log R)\frac{\varphi (W)}{W}\left( \prod _{i=1}^{k}\frac{\varphi (u_{i})}{u_{i}} \right)F_{u_{1},\cdots ,u_{k}}^{(m)}+O\left( \frac{F_{\mathrm{max}}\varphi (W)\log R}{WD_{0}} \right),
\end{equation}
where 
\begin{equation}
\label{38}
F_{u_{1},\cdots ,u_{k}}^{[m]}=\int _{0}^{1}F\left( \frac{\log u_{1}}{\log R}, \cdots ,\frac{\log u_{m-1}}{\log R}, v, \frac{\log u_{m+1}}{\log R}, \cdots ,\frac{\log u_{k}}{\log R} \right) dv.
\end{equation}
Combining (\ref{36}), (\ref{37}), we obtain 
\begin{equation}
\label{39}
\begin{aligned}
y_{u_{1},\cdots ,u_{k}}^{(m)}(p)y_{u_{1},\cdots ,u_{k}}^{(m)}=&-\frac{\varphi (W)^{2}}{pW^{2}}(\log R)\left(\log \frac{R}{p}\right)\left( \prod _{i\neq m}\frac{\varphi (u_{i})^{2}}{u_{i}^{2}}\right) F_{u_{1},\cdots ,u_{k}}^{(p;m)}F_{u_{1},\cdots ,u_{k}}^{[m]} \\
&\quad +O\left( \frac{\varphi (W)^{2}F_{\mathrm{max}}^{2}\log ^{2}R}{pW^{2}D_{0}} \right)
\end{aligned}
\end{equation}
if $u_{m}=1$. In the above computation, we used the trivial estimates 
\[
\sup _{u_{1},\cdots ,u_{k}}|F_{u_{1},\cdots ,u_{k}}^{[m]}| \ll F_{\mathrm{max}},\quad \sup _{u_{1},\cdots ,u_{k}}|F_{u_{1},\cdots ,u_{k}}^{(p;m)}| \ll F_{\mathrm{max}}.
\]
We substitute this into the sum over $u_{1}, \cdots ,u_{k}$ in (\ref{26}). The contribution of the error term is 
\[
\ll \frac{\varphi (W)^{2}F_{\mathrm{max}}^{2}\log ^{2}R}{pW^{2}D_{0}}\underset{(u_{i}, W)=1 (\forall i)}{\sum _{u_{1},\cdots ,u_{m-1},u_{m+1},\cdots ,u_{k}< R}}\frac{1}{\prod _{i=1}^{k}g(u_{i})}\ll \frac{\varphi (W)^{k+1}F_{\mathrm{max}}^{2}\log ^{k+1}N}{pW^{k+1}D_{0}}.
\]
Therefore, 
\begin{equation}
\label{40}
\begin{aligned}
&\sum _{u_{1}, \cdots ,u_{k}}\frac{y_{u_{1},\cdots ,u_{k}}^{(m)}(p)y_{u_{1},\cdots ,u_{k}}^{(m)}}{\prod _{i=1}^{k}g(u_{i})} \\
& \quad =-\frac{\varphi (W)^{2}}{pW^{2}}(\log R)\left( \log \frac{R}{p} \right)
\underset{(u_{i},pW)=1 (\forall i)}{\underset{(u_{i},u_{j})=1 (\forall i\neq j)}{\underset{u_{m}=1}{\sum _{u_{1}, \cdots ,u_{k}}}}}\left( \prod _{i\neq m}\frac{\mu ^{2} (u_{i})\varphi (u_{i})^{2}}{u_{i}^{2}g(u_{i})}\right)F_{u_{1},\cdots ,u_{k}}^{(p;m)}F_{u_{1},\cdots ,u_{k}}^{[m]} \\
&\quad \quad \quad  +O\left( \frac{\varphi (W)^{k+1}F_{\mathrm{max}}^{2}\log ^{k+1}N}{pW^{k+1}D_{0}} \right).
\end{aligned}
\end{equation}
We compute the sum over $u_{1}, \cdots ,u_{k}$ in (\ref{40}). First,  we remove the condition that $u_{i}$ and $u_{j}$ are coprime if $i\neq j$.  Since $(u_{i}, W)=(u_{j},W)=1$, if $(u_{i},u_{j})>1$, there exists a prime $q> D_{0}$ such that $q|u_{i}, u_{j}$. Therefore, the possible error is at most
\[
F_{\mathrm{max}}^{2}\left( \sum _{q>D_{0}}\frac{\varphi (q)^{4}}{g(q)^{2}q^{4}}\right)\left( \underset{(u,W)=1}{\sum _{u<R}}\frac{\varphi (u)^{2}}{u^{2}g(u)}\right)^{k-1}\ll \frac{F_{\mathrm{max}}^{2}\varphi (W)^{k-1}(\log R)^{k-1}}{D_{0}W^{k-1}}.
\]
Hence we have
\begin{equation}
\label{41}
\begin{aligned}
&\underset{(u_{i},pW)=1 (\forall i)}{\underset{(u_{i},u_{j})=1 (\forall i\neq j)}{\underset{u_{m}=1}{\sum _{u_{1}, \cdots ,u_{k}}}}}\left( \prod _{i\neq m}\frac{\mu ^{2} (u_{i})\varphi (u_{i})^{2}}{u_{i}^{2}g(u_{i})}\right)F_{u_{1},\cdots ,u_{k}}^{(p;m)}F_{u_{1},\cdots ,u_{k}}^{[m]} \\
&\quad = \underset{(u_{i},pW)=1 (\forall i)}{\underset{u_{m}=1}{\sum _{u_{1}, \cdots ,u_{k}}}}\left( \prod _{i\neq m}\frac{\mu ^{2} (u_{i})\varphi (u_{i})^{2}}{u_{i}^{2}g(u_{i})}\right)F_{u_{1},\cdots ,u_{k}}^{(p;m)}F_{u_{1},\cdots ,u_{k}}^{[m]} \\
&\quad\quad\quad +O\left( \frac{F_{\mathrm{max}}^{2}\varphi (W)^{k-1}(\log R)^{k-1}}{D_{0}W^{k-1}} \right).
\end{aligned}
\end{equation}
Now we apply Lemma 3.4 with 
\begin{eqnarray}
\gamma (q)=\left\{ \begin{array}{ll}
 1-\frac{q^{2}-3q+1}{q^{3}-q^{2}-2q+1} & (q\mid \hspace{-.67em}/ pW) \\
 0 & (\mathrm{otherwise}) \\
\end{array} \right.
\end{eqnarray} 
to the sum over $u_{1}, \cdots . u_{m-1}, u_{m+1},\cdots ,u_{k}$. Then we have
\begin{align*}
&\underset{(u_{i},pW)=1 (\forall i)}{\underset{u_{m}=1}{\sum _{u_{1}, \cdots ,u_{k}}}}\left( \prod _{i\neq m}\frac{\mu ^{2} (u_{i})\varphi (u_{i})^{2}}{u_{i}^{2}g(u_{i})}\right)F_{u_{1},\cdots ,u_{k}}^{(p;m)}F_{u_{1},\cdots ,u_{k}}^{[m]} \\
&=\mathfrak{S} ^{k-1}(\log R)^{k-1} \\
&\quad \times \int _{0}^{1}\cdots \int _{0}^{1}\left\{ \int _{0}^{1} F_{p}^{[m]}(u_{1},\cdots ,u_{m-1},u,u_{m+1}, \cdots ,u_{k})du \right\} \\
&\quad\quad\quad\quad \times \left\{ \int _{0}^{1} F(u_{1},\cdots ,u_{m-1},u,u_{m+1}, \cdots ,u_{k})du \right\} du_{1}\cdots du_{m-1}du_{m+1}\cdots du_{k} \\
&\quad\quad +O(\mathfrak{S} ^{k-1}L (\log R)^{k-2}F_{\mathrm{max}}^{2}).
\end{align*}
In this case, 
\[
L \ll 1+\sum _{q|pW}\frac{\log q}{q} \ll \log D_{0},
\]
\begin{align*}
\mathfrak{S} =& \left\{ \prod_{q \mid \hspace{-.33em}/ pW}\left( 1-\frac{1}{q}+O\left( \frac{1}{q^{2}}\right) \right)^{-1}\left( 1-\frac{1}{q}\right) \right\} \prod _{q|pW}\left( 1-\frac{1}{q} \right) \\
=&\frac{\varphi (p)\varphi (W)}{pW}\underset{q\neq p}{\prod _{q\geq D_{0}}}\left(1 +O\left( \frac{1}{q^{2}}\right) \right) \\
=&\frac{\varphi (p)\varphi (W)}{pW}(1+O(D_{0}^{-1})).
\end{align*}
Therefore, 
\begin{equation}
\label{42}
\begin{aligned}
&\underset{(u_{i},pW)=1 (\forall i)}{\underset{u_{m}=1}{\sum _{u_{1}, \cdots ,u_{k}}}}\left( \prod _{i\neq m}\frac{\mu ^{2} (u_{i})\varphi (u_{i})^{2}}{u_{i}^{2}g(u_{i})}\right)F_{u_{1},\cdots ,u_{k}}^{(p;m)}F_{u_{1},\cdots ,u_{k}}^{[m]} \\&\quad =\frac{\varphi (p)^{k-1}\varphi (W)^{k-1}}{p^{k-1}W^{k-1}}(\log R)^{k-1} \\
&\quad\quad \times \int _{0}^{1}\cdots \int _{0}^{1}\left\{ \int _{0}^{1} F_{p}^{[m]}(u_{1},\cdots ,u_{m-1},u,u_{m+1}, \cdots ,u_{k})du \right\} \\
&\quad\quad\quad \times \left\{ \int _{0}^{1} F(u_{1},\cdots ,u_{m-1},u,u_{m+1}, \cdots ,u_{k})du \right\} du_{1}\cdots du_{m-1}du_{m+1}\cdots du_{k} \\
& \quad\quad\quad\quad +O\left( \frac{\varphi (p)^{k-1}\varphi (W)^{k-1}}{D_{0}p^{k-1}W^{k-1}}(\log R)^{k-1}F_{\mathrm{max}}^{2}\right).
\end{aligned}
\end{equation}
We put 
\begin{equation}
\label{43}
\begin{aligned}
J_{k}^{(m)}[p]=&\int _{0}^{1}\cdots \int _{0}^{1}\left\{ \int _{0}^{1} F_{p}^{[m]}(u_{1},\cdots ,u_{m-1},u,u_{m+1}, \cdots ,u_{k})du \right\} \\
&\times \left\{ \int _{0}^{1} F(u_{1},\cdots ,u_{m-1},u,u_{m+1}, \cdots ,u_{k})du \right\} du_{1}\cdots du_{m-1}du_{m+1}\cdots du_{k}.
\end{aligned}
\end{equation}
By substituting  (\ref{42}), (\ref{43}) into (\ref{41}), we have
\begin{equation}
\label{44}
\begin{aligned}
& \underset{(u_{i},pW)=1 (\forall i)}{\underset{(u_{i},u_{j})=1 (\forall i\neq j)}{\underset{u_{m}=1}{\sum _{u_{1}, \cdots ,u_{k}}}}}\left( \prod _{i\neq m}\frac{\mu ^{2} (u_{i})\varphi (u_{i})^{2}}{u_{i}^{2}g(u_{i})}\right)F_{u_{1},\cdots ,u_{k}}^{(p;m)}F_{u_{1},\cdots ,u_{k}}^{[m]} \\
&\quad\quad = \frac{\varphi (p)^{k-1}\varphi (W)^{k-1}}{p^{k-1}W^{k-1}}(\log R)^{k-1}J_{k}^{(m)}[p] +O\left( \frac{F_{\mathrm{max}}^{2}\varphi (W)^{k-1}(\log R)^{k-1}}{D_{0}W^{k-1}} \right).
\end{aligned}
\end{equation}
We substitute (\ref{44}) into (\ref{40}). Since $\log R \ll \log N$, we obtain
\begin{equation}
\label{45}
\begin{aligned}
&\sum _{u_{1}, \cdots ,u_{k}}\frac{y_{u_{1},\cdots ,u_{k}}^{(m)}(p)y_{u_{1},\cdots ,u_{k}}^{(m)}}{\prod _{i=1}^{k}g(u_{i})}\\
&\quad=-\frac{\varphi (p)^{k-1}\varphi (W)^{k+1}}{p^{k}W^{k+1}}(\log R)^{k}\left( \log \frac{R}{p} \right)J_{k}^{(m)}[p] +O\left( \frac{F_{\mathrm{max}}^{2}\varphi (W)^{k+1}(\log N)^{k+1}}{pD_{0}W^{k+1}} \right) \\
&\quad =-\frac{ \varphi (W)^{k+1}}{pW^{k+1}}(\log R)^{k}\left( \log \frac{R}{p} \right)J_{k}^{(m)}[p] +O\left( \frac{F_{\mathrm{max}}^{2}\varphi (W)^{k+1}(\log N)^{k+1}}{pD_{0}W^{k+1}} \right). \\
\end{aligned}
\end{equation}
We substitute this into (\ref{26}). We compute the sum over $p$. Recall that $J_{k}^{(m)}[p]$ is defined by (\ref{43}), and the function $F_{p}^{[m]}$ is obtained by replacing the $m$-th component $u$ of $F$ with $((\log R/p)/\log R)(\log p/(\log R/p) +u)$. To see how the sum over $p$  becomes,  let us compute the sum 
\begin{equation}
\label{46}
\sum _{Y<p< R}\frac{1}{p}\left( \log \frac{R}{p}\right)\left( \log \frac{N}{p} \right)^{-1}f\left( \frac{\log \frac{R}{p}}{\log R}\left( \frac{\log p}{\log \frac{R}{p}}+u \right) \right)
\end{equation}
for any smooth function $f$, where $Y=N^{\eta}, R=N^{\frac{\theta}{2}-\delta}$. We denote by $\pi (v)$ the number of primes equal or less than $v$. Then, the sum (\ref{46}) is expressed by 
\[
\int _{N^{\eta}}^{N^{\frac{\theta}{2}-\delta}}\frac{1}{v}\left( \log \frac{R}{v}\right) \left( \log \frac{N}{v}\right)^{-1}f\left( \frac{\log \frac{R}{v}}{\log R}\left( \frac{\log v}{\log \frac{R}{v}}+u \right) \right)d\pi (v).
\]
Using the Prime Number Theorem, this is asymptotically
\[
\int _{N^{\eta}}^{N^{\frac{\theta}{2}-\delta}}\frac{1}{v}\left( \log \frac{R}{v}\right) \left( \log \frac{N}{v}\right)^{-1}f\left( \frac{\log \frac{R}{v}}{\log R}\left( \frac{\log v}{\log \frac{R}{v}}+u \right) \right)\frac{dv}{\log v}.
\]
By putting $\log v/\log N=\xi$, this becomes 
\[
\int _{\eta}^{\frac{\theta}{2}-\delta}\frac{\frac{\theta}{2}-\delta -\xi}{1-\xi}f\left( \frac{\xi}{\frac{\theta}{2}-\delta}+\frac{\frac{\theta}{2}-\delta -\xi}{\frac{\theta}{2}-\delta}u \right)\frac{d\xi}{\xi}.
\]
We put
\begin{equation}
\label{47}
\begin{aligned}
&F_{m, \delta}(u_{1}, \cdots ,u_{k}; \xi) \\
&\quad:=F\left( u_{1},\cdots ,u_{m-1}, \frac{\xi}{\frac{\theta}{2}-\delta}+\frac{\frac{\theta}{2}-\delta -\xi}{\frac{\theta}{2}-\delta}u_{m}, u_{m+1},\cdots ,u_{k} \right), 
\end{aligned}
\end{equation}
\begin{equation}
\label{48}
\begin{aligned}
L_{k, \delta}^{[m]}(\xi):=& \int _{0}^{1}\cdots \int _{0}^{1}\left\{ \int _{0}^{1} F_{m,\delta}(u_{1}, \cdots ,u_{k};\xi)du_{m} \right\} \\
&\quad\quad \quad\quad \times \left\{ \int _{0}^{1}F(u_{1},\cdots ,u_{k})du_{m}\right\}du_{1}\cdots du_{m-1}du_{m+1}\cdots du_{k}.
\end{aligned}
\end{equation}
Then, by the above argument and simple estimate
\[
\sum _{Y<p< R}\frac{1}{p\log \frac{N}{p}}\ll _{\eta}\frac{1}{\log N} \quad (Y=N^{\eta}, R=N^{\frac{\theta}{2}-\delta}),
\]
we have
\begin{equation}
\label{48.5}
\begin{aligned}
&\sum _{Y<p < R}\frac{1}{\log \frac{N}{p}}\sum _{u_{1}, \cdots ,u_{k}}\frac{y_{u_{1},\cdots ,u_{k}}^{(m)}(p)y_{u_{1},\cdots ,u_{k}}^{(m)}}{\prod _{i=1}^{k}g(u_{i})} \\
&\quad =-\frac{\varphi (W)^{k+1}}{W^{k+1}}(\log R)^{k}(1+o(1))\int _{\eta}^{\frac{\theta}{2}-\delta}\frac{\frac{\theta}{2}-\delta -\xi}{1-\xi}L_{k,\delta}^{[m]}(\xi)\frac{d\xi}{\xi} \\
&\quad\quad\quad +O\left( \frac{F_{\mathrm{max}}^{2}\varphi (W)^{k+1}(\log N)^{k}}{D_{0}W^{k+1}} \right),
\end{aligned}
\end{equation}
if the integral of the main term is not zero. Finally, by substituting  this into (\ref{26}), and  combining (\ref{31}), (\ref{32}) and
\[
y_{\mathrm{max}}^{(m)}\ll \frac{\varphi (W)}{W}F_{\mathrm{max}}\log N,\quad  y_{\mathrm{max}}\ll F_{\mathrm{max}}
\]
(see \cite{M1}, p.403), we obtain the following result: 
\begin{prop}
Assume $BV[\theta ,{\cal P}]$. Then, if
\[
L_{k,\delta}^{(m)}(F):=\int _{\eta}^{\frac{\theta}{2}-\delta}\frac{\frac{\theta}{2}-\delta -\xi}{1-\xi}L_{k,\delta}^{[m]}(\xi)\frac{d\xi}{\xi} \neq 0, 
\]
we have
\begin{equation}
\label{49}
\begin{aligned}
S_{2,I}^{(m)}=&-\frac{\varphi (W)^{k}N}{W^{k+1}}(\log R)^{k}(1+o(1))L_{k,\delta}^{(m)}(F) \\
&\quad\quad +O\left( \frac{F_{\mathrm{max}}^{2}\varphi (W)^{k}N(\log N)^{k}}{D_{0}W^{k+1}} \right) +O_{B}\left( \frac{NF_{\mathrm{max}}^{2}}{(\log N)^{B}} \right)
\end{aligned}
\end{equation}
as $N\to \infty$, where
\begin{align*}
L_{k, \delta}^{[m]}(\xi):=& \int _{0}^{1}\cdots \int _{0}^{1}\left\{ \int _{0}^{1} F_{m,\delta}(u_{1}, \cdots ,u_{k};\xi)du_{m} \right\} \\
&\quad\quad \quad\quad \times \left\{ \int _{0}^{1}F(u_{1},\cdots ,u_{k})du_{m}\right\}du_{1}\cdots du_{m-1}du_{m+1}\cdots du_{k},
\end{align*}
\begin{align*}
&F_{m, \delta}(u_{1}, \cdots ,u_{k}; \xi) \\
&\quad\quad\quad :=F\left( u_{1},\cdots ,u_{m-1}, \frac{\xi}{\frac{\theta}{2}-\delta}+\frac{\frac{\theta}{2}-\delta -\xi}{\frac{\theta}{2}-\delta}u_{m}, u_{m+1},\cdots ,u_{k} \right).
\end{align*}
\end{prop}
Notice that the same result holds for $S_{2, I\hspace{-.em}I}^{(m)}$. If $L_{k,\delta}^{(m)}(F)=0$, the leading term vanishes and hence
\[
S_{2,I}^{(m)}=S_{2, I\hspace{-.em}I}^{(m)}=o \left( \frac{F_{\mathrm{max}}^{2}\varphi (W)^{k}N(\log N)^{k}}{W^{k+1}} \right)+O_{B}\left( \frac{NF_{\mathrm{max}}^{2}}{(\log N)^{B}} \right).
\]

%%%%%%%%%%%%%%%%%%%%%%%%%%%%%%%%%%%%%%%%%%%%%%%%%%%%%%%%%%%%%%%%%%%%%%%%%%%%%%%%%%%%%%%%%%%%%%%%%%%%%%%%%%%%%%%%%%%%%%%%%%%%%%%%%%%%%%%%%%%%%%%%%%%%%%%%%%%%%%%%%%%%%%%%%%%%%%%%%%%%%%%%%%%%%%%%%%%%%%%%%%%%%%%%%%%%%%%%%%%%%%%%%%%%%%%%%%%%%%%%
\section{The computation of $S_{2, I\hspace{-.em}I\hspace{-.em}I}^{(m)}$}
To compute
\begin{equation}
\label{55}
S_{2, I\hspace{-.em}I\hspace{-.em}I}^{(m)}:=\underset{d_{m}=e_{m}=1}{\underset{e_{1},\cdots ,e_{k}}{\sum _{d_{1}, \cdots ,d_{k}}}}\lambda _{d_{1}, \cdots ,d_{k}}\lambda _{e_{1}, \cdots ,e_{k}}\underset{[d_{i}, e_{i}]|n+h_{i} (\forall i)}{\underset{n\equiv \nu _{0} (\mathrm{mod}\; W)}{\sum _{N\leq n<2N}}}\beta (n+h_{m}), 
\end{equation}
we use the following lemma: 
\begin{lem}
Let $\beta (n)$ be the function defined by 
\begin{eqnarray}
\beta (n)=
\left\{
\begin{array}{l}
1 \quad (n=p_{1}p_{2}, Y<p_{1}\leq N^{\frac{1}{2}}<p_{2}) \\
0 \quad (\mathrm{otherwise}),
\end{array}
\right.
\end{eqnarray}
where $Y=N^{\eta}$, $1\ll \eta <\frac{1}{4}$. Then, we have
\begin{equation}
\label{50}
\underset{(n,q)=1}{\sum _{N\leq n<2N}}\beta (n)=\frac{N}{\log N}\log \frac{1-\eta}{\eta}+O \left( \frac{N\log \log N}{(\log N)^{2}} \right)
\end{equation}
uniformly for $q\leq N$. Here, the implicit constant might be dependent on $\eta$.
\end{lem}
\begin{proof}
We denote by $\omega (q)$ the number of distinct prime factors of $q$. Then,
\begin{align*}
\underset{(n,q)=1}{\sum _{N\leq n<2N}}\beta (n)=& \underset{(p_{1},q)=1}{\sum _{Y<p_{1}\leq N^{\frac{1}{2}}}}\underset{(p_{2},q)=1}{\sum _{\frac{N}{p_{1}}\leq p_{2}<\frac{2N}{p_{1}}}}1 \\
=& \underset{(p_{1},q)=1}{\sum _{Y<p_{1}\leq N^{\frac{1}{2}}}}\left\{ \pi ^{\flat}\left( \frac{N}{p_{1}}\right) +O(\omega (q)) \right\} \\
=&  \underset{(p_{1},q)=1}{\sum _{Y<p_{1}\leq N^{\frac{1}{2}}}}\pi ^{\flat}\left( \frac{N}{p_{1}} \right) +O(N^{\frac{1}{2}}\omega (q)) \\
=&\sum _{Y<p_{1}\leq N^{\frac{1}{2}}}\pi ^{\flat}\left( \frac{N}{p_{1}} \right)+O\left( \pi ^{\flat}\left( \frac{N}{Y}\right)\omega (q) \right)+O(N^{\frac{1}{2}}\omega (q)).
\end{align*}
By applying the Prime Number Theorem to the final line, we obtain
\begin{equation}
\label{51}
\begin{aligned}
\underset{(n,q)=1}{\sum _{N\leq n<2N}}\beta (n)=\sum _{Y<p_{1}\leq N^{\frac{1}{2}}}\left\{ \frac{N}{p_{1}\log \frac{N}{p_{1}}}+O\left( \frac{N}{p_{1}(\log N)^{2}}\right) \right\} +O\left( \omega (q)\frac{N^{1-\eta}}{\log N} \right).
\end{aligned}
\end{equation}
The contribution of the first error term is 
\[
\sum _{Y<p_{1}\leq N^{\frac{1}{2}}}\frac{N}{p_{1}(\log N)^{2}}\ll \frac{N\log \log N}{(\log N)^{2}}.
\]
On the other hand, the main term is 
\begin{equation}
\label{52}
\begin{aligned}
\sum _{Y<p_{1}\leq N^{\frac{1}{2}}}\frac{N}{p_{1}\log \frac{N}{p_{1}}} =&N\int _{Y}^{N^{\frac{1}{2}}}\frac{d\pi (u)}{u(\log N-\log u)} \\
=& \frac{N}{\log N}\sum _{k=0}^{\infty}\frac{1}{\log ^{k}N}\int _{Y}^{N^{\frac{1}{2}}}\frac{\log ^{k}u}{u}d\pi (u).
\end{aligned}
\end{equation}
By partial integration, we find that
\begin{align*}
\int_{Y}^{N^{\frac{1}{2}}}\frac{d\pi (u)}{u}=&\left[ \frac{1}{\log u}+O\left( \frac{1}{\log ^{2}u}\right) \right]_{N^{\eta}}^{N^{\frac{1}{2}}}+\int _{N^{\eta}}^{N^{\frac{1}{2}}}\frac{1}{u^{2}}\left( \frac{u}{\log u}+O\left( \frac{u}{\log ^{2}u} \right)\right)du \\
=&\log \left( \frac{1}{2\eta} \right) +O\left( \frac{1}{\log N} \right),  
\end{align*}
\begin{align*}
\int _{Y}^{N^{\frac{1}{2}}}\frac{\log u}{u}d\pi (u)=& \left[ 1+O\left( \frac{1}{\log u}\right) \right]_{N^{\eta}}^{N^{\frac{1}{2}}} \\
& -\int _{N^{\eta}}^{N^{\frac{1}{2}}}\frac{1-\log u}{u^{2}}\left( \frac{u}{\log u}+O\left( \frac{u}{\log ^{2}u }\right)\right) du \\
=& \left( \frac{1}{2}-\eta \right)\log N+O(1),
\end{align*}
and for $k\geq 2$, we have
\begin{align*}
\int _{Y}^{N^{\frac{1}{2}}}\frac{\log ^{k}u}{u}d\pi (u)=& \left[ \frac{\log ^{k}u}{u} \left( \frac{u}{\log u} +O\left( \frac{u}{\log ^{2}u} \right) \right) \right] _{N^{\eta}}^{N^{\frac{1}{2}}} \\
& -\int _{N^{\eta}}^{N^{\frac{1}{2}}}\frac{k \log ^{k-1}u -\log ^{k}u}{u^{2}}\left( \frac{u}{\log u} +O\left( \frac{u}{\log ^{2}u} \right) \right)du \\
=& \frac{1}{k}\left(\frac{1}{2^{k}}-\eta ^{k} \right)\log ^{k}N+O(2^{-k}\log ^{k-1}N).
\end{align*}
(The implicit constant might be dependent on $\eta$, but independent of $k$.) Combining these, we obtain 
\begin{equation}
\label{53}
\begin{aligned}
\sum _{k=0}^{\infty} \frac{1}{\log ^{k}N}\int _{Y}^{N^{\frac{1}{2}}}\frac{\log ^{k}u}{u}d\pi (u)&=\log \left(\frac{1}{2\eta} \right)+\sum _{k=1}^{\infty}\frac{1}{k}\left( \frac{1}{2^{k}}-\eta ^{k} \right) +O\left( \sum _{k=1}^{\infty}\frac{2^{-k}}{\log N} \right) \\
&=\log \frac{1-\eta}{\eta}+O\left( \frac{1}{\log N} \right).
\end{aligned}
\end{equation}
By substituting  (\ref{53}) into (\ref{52}), we have
\[
\sum _{Y<p_{1}\leq N^{\frac{1}{2}}}\frac{N}{p_{1}\log \frac{N}{p_{1}}}=\frac{N}{\log N}\log \frac{1-\eta}{\eta} +O\left( \frac{N}{\log ^{2}N} \right).
\]
Consequently, 
\begin{equation}
\label{54}
\underset{(n,q)=1}{\sum _{N\leq n<2N}}\beta (n)=\frac{N}{\log N}\log \frac{1-\eta}{\eta}+O\left( \frac{N\log \log N}{\log ^{2}N} \right)+O\left( \frac{\omega (q)N^{1-\eta}}{\log N} \right).
\end{equation}
Since $\omega (q) \ll N^{\epsilon}$ ($\forall \epsilon >0$) holds uniformly for $q\leq N$, the second error term is dominated by the first one. Hence we obtain the result. 
\end{proof}
We return to the computation of $S_{2, I\hspace{-.em}I\hspace{-.em}I}^{(m)}$, defined by (\ref{55}).  The only pairs $(d_{1},\cdots ,d_{k},e_{1},\cdots ,e_{k})$ satisfying the condition that $W, [d_{1},e_{1}], \cdots, [d_{k}, e_{k}]$ are pairwise coprime contribute to the sum. We denote the restricted sum by  $\sum ^{'}$. We put
\[
q=W\prod _{i=1}^{k}[d_{i},e_{i}].
\]
Then, there exists a unique $\nu \; (\mathrm{mod}\; q)$ such that $\nu \equiv \nu _{0}\; (\mathrm{mod} \; W)$, $h_{i}+\nu \equiv 0\; (\mathrm{mod} \; [d_{i},e_{i}])$ ($i=1, \cdots ,k$) and the sum over $n$ is rewritten as the sum over integers congruent to $\nu$ modulo $q$. Therefore, 
\begin{equation}
\label{56}
\begin{aligned}
\underset{[d_{i},e_{i}]|n+h_{i} (\forall i)}{\underset{n\equiv \nu _{0} (\mathrm{mod}\; W)}{\sum_{N\leq n<2N}}}\beta (n+h_{m})=& \underset{n \equiv \nu \; (\mathrm{mod}\; q)}{\sum _{N\leq n<2N}}\beta (n+h_{m}) \\
=& \underset{n \equiv \nu ^{'} \; (\mathrm{mod}\; q)}{\sum _{N<n\leq 2N}}\beta (n)+O(1),
\end{aligned}
\end{equation}
where $\nu ^{'}=\nu +h_{m}$. This $\nu ^{'}$ satisfies $(\nu ^{'}, q)=1$. This fact follows from our choice of $\nu _{0}$ and the condition that the elements of $\cal{H}$ are bounded.  We have treated the similar situation in Section 3, hence we omit to prove this. Hence by (\ref{56}), we have
\begin{equation}
\label{57}
\begin{aligned}
\underset{[d_{i},e_{i}]| n+h_{i} (\forall i)}{\underset{n\equiv \nu _{0}\; (\mathrm{mod}\; W)}{\sum _{N\leq n<2N}}}\beta (n+h_{m})&=\frac{1}{\varphi (q)}\underset{(n,q)=1}{\sum _{N\leq n<2N}}\beta (n) +\Delta _{\beta}(N; q, \nu ^{'}) +O(1),
\end{aligned}
\end{equation}
where 
\[
\Delta _{\beta}(N; q, \nu ^{'})=\underset{n\equiv \nu ^{'}\; (\mathrm{mod}\; q)}{\sum _{N\leq n<2N}}\beta (n) -\frac{1}{\varphi (q)}\underset{(n,q)=1}{\sum _{N\leq n<2N}}\beta (n).
\]
Now we apply Lemma 4.1 to the sum in (\ref{57}). Then we have
\begin{align*}
\underset{[d_{i},e_{i}]| n+h_{i} (\forall i)}{\underset{n\equiv \nu _{0}\; (\mathrm{mod}\; W)}{\sum _{N\leq n<2N}}} \beta (n+h_{m})= \frac{X_{N, \eta}}{\varphi (q)}  +\Delta _{\beta}(N; q, \nu ^{'})+O(1),
\end{align*}
where 
\begin{equation}
\label{defofxne}
X_{N, \eta}=\frac{N}{\log N}\log \frac{1-\eta}{\eta}+O\left( \frac{N\log \log N}{(\log N)^{2}} \right).
\end{equation}
By substituting  this into (\ref{55}), we obtain 
\begin{equation}
\label{58}
\begin{aligned}
S_{2, I\hspace{-.em}I\hspace{-.em}I}^{(m)}=& \frac{X_{N, \eta}}{\varphi (W)}\underset{d_{m}=e_{m}=1}{\underset{e_{1},\cdots ,e_{k}}{\sum ^{\quad\quad '}_{d_{1}, \cdots ,d_{k}}}} \frac{\lambda _{d_{1},\cdots ,d_{k}}\lambda _{e_{1},\cdots ,e_{k}}}{\prod _{i=1}^{k}\varphi ([d_{i}, e_{i}])} \\
&\quad +O\left(\underset{e_{1}, \cdots ,e_{k}}{\sum ^{\quad\quad '} _{d_{1} ,\cdots ,d_{k}}}|\lambda _{d_{1},\cdots ,d_{k}}\lambda _{e_{1} ,\cdots ,e_{k}}|(|\Delta _{\beta}(N; q,\nu ^{'})|+1 ) \right).
\end{aligned}
\end{equation}
Under the assumption of the estimation $BV[\theta , {\cal E}_{2}]$, the error term above is  evaluated by $O_{B}(Ny_{\mathrm{max}}^{2}/(\log N)^{B})$. The proof of this statement is essentially the same as that in \cite{M1}, hence we omit it. Moreover, the sum in the main term is also computed in \cite{M1} (see the proof of Lemma 5.2 of \cite{M1}). The result is 
\[
\underset{d_{m}=e_{m}=1}{\underset{e_{1},\cdots ,e_{k}}{\sum ^{\quad\quad '}_{d_{1}, \cdots ,d_{k}}}} \frac{\lambda _{d_{1},\cdots ,d_{k}}\lambda _{e_{1},\cdots ,e_{k}}}{\prod _{i=1}^{k}\varphi ([d_{i}, e_{i}])} =\sum _{u_{1},\cdots ,u_{k}}\frac{(y_{u_{1}, \cdots ,u_{k}}^{(m)})^{2}}{\prod _{i=1}^{k}g(u_{i})} +O\left( \frac{(y_{\mathrm{max}}^{(m)})^{2}\varphi (W)^{k-1}(\log N)^{k-1}}{D_{0}W^{k-1}} \right).
\]
Consequently, we obtain 
\begin{equation}
\label{59}
\begin{aligned}
S_{2, I\hspace{-.em}I\hspace{-.em}I}^{(m)}=& \frac{X_{N, \eta}}{\varphi (W)}\sum _{u_{1}, \cdots ,u_{k}}\frac{(y_{u_{1}, \cdots ,u_{k}}^{(m)})^{2}}{\prod _{i=1}^{k}g(u_{i})} \\
&\quad +O\left( \frac{(y_{\mathrm{max}}^{(m)})^{2}\varphi (W)^{k-2}N(\log N)^{k-2}}{D_{0}W^{k-1}} \right) +O_{B}\left( \frac{Ny_{\mathrm{max}}^{2}}{(\log N)^{B}}\right).
\end{aligned}
\end{equation}
We put
\[
J_{k}^{(m)}(F)=\int _{0}^{1}\cdots \int _{0}^{1}\left( \int _{0}^{1}F(t_{1},\cdots ,t_{k})dt_{m}\right)^{2}dt_{1}\cdots dt_{m-1}dt_{m+1}\cdots dt_{k}.
\]
By Lemma 6.3 of \cite{M1}, we have
\begin{align*}
& \frac{N}{\varphi (W)\log N}\sum _{u_{1}, \cdots ,u_{k}}\frac{(y_{u_{1}, \cdots ,u_{k}}^{(m)})^{2}}{\prod _{i=1}^{k}g(u_{i})} \\
& \quad =\frac{\varphi (W)^{k}N(\log R)^{k+1}}{W^{k+1}\log N}J_{k}^{(m)}(F)+O\left( \frac{F_{\mathrm{max}}^{2}\varphi (W)^{k}N(\log N)^{k}}{W^{k+1}D_{0}} \right).
\end{align*}
Moreover, we have
\[
y_{\mathrm{max}} \ll F_{\mathrm{max}}, \quad y_{\mathrm{max}}^{(m)} \ll \frac{F_{\mathrm{max}}\varphi (W)\log N}{W}
\]
(see \cite{M1}, p.403). By substituting the definition of $X_{N, \eta}$ (\ref{defofxne}) into (\ref{58}) and combining these, we obtain the following result:
\begin{prop}
Under the assumption of $BV[\theta ,{\cal E}_{2}]$, we have
\begin{equation}
\label{60}
\begin{aligned}
S_{2, I\hspace{-.em}I\hspace{-.em}I}^{(m)}=& \frac{\varphi (W)^{k}N(\log R)^{k+1}}{W^{k+1}\log N}\left( \log \frac{1-\eta}{\eta} \right)(1+o(1))J_{k}^{(m)}(F) \\
 &\quad +O\left( \frac{F_{\mathrm{max}}^{2}\varphi (W)^{k}N(\log N)^{k}}{W^{k+1}D_{0}} \right) +O_{B}\left( \frac{F_{\mathrm{max}}^{2}N}{(\log N)^{B}} \right)
\end{aligned}
\end{equation}
as $N\to \infty$, where 
\begin{equation}
\label{61}
J_{k}^{(m)}(F)=\int _{0}^{1}\cdots \int _{0}^{1}\left( \int _{0}^{1}F(t_{1},\cdots ,t_{k})dt_{m}\right)^{2}dt_{1}\cdots dt_{m-1}dt_{m+1}\cdots dt_{k}.
\end{equation}
\end{prop}

%%%%%%%%%%%%%%%%%%%%%%%%%%%%%%%%%%%%%%%%%%%%%%%%%%%%%%%%%%%%%%%%%%%%%%%%%%%%%%%%%%%%%%%%%%%%%%%%%%%%%%%%%%%%%%%%%%%%%%%%%%%%%%%%%%%%%%%%%%%%%%%%%%%%%%%%%%%%%%%%
\section{The computation of $S_{2, I\hspace{-.em}V}^{(m)}$} 
The next problem is to compute
\begin{equation}
\label{62}
S_{2, I\hspace{-.em}V}^{(m)}:=\sum _{Y<p< R}\underset{d_{m}=e_{m}=p}{\underset{e_{1},\cdots ,e_{k}}{\sum _{d_{1}, \cdots ,d_{k}}}}\lambda _{d_{1}, \cdots ,d_{k}}\lambda _{e_{1}, \cdots ,e_{k}}\underset{[d_{i}, e_{i}]|n+h_{i} (\forall i)}{\underset{n\equiv \nu _{0} (\mathrm{mod}\; W)}{\sum _{N\leq n<2N}}}\beta (n+h_{m}).
\end{equation}
The only pairs $(d_{1}, \cdots ,d_{k}, e_{1}, \cdots ,e_{k})$ satisfying the condition that $W, [d_{1}, e_{1}]$, $\cdots$, $[d_{k},e_{k}]$ are pairwise coprime contribute to the sum above. We put
\[
q=W\prod _{i=1}^{k}[d_{i}, e_{i}].
\]
Then, 
\begin{equation}
\label{63}
\underset{[d_{i}, e_{i}]|n+h_{i} (\forall i)}{\underset{n\equiv \nu _{0} (\mathrm{mod}\; W)}{\sum _{N\leq n<2N}}}\beta (n+h_{m})=\underset{n \equiv \nu \; (\mathrm{mod}\; q)}{\sum _{N\leq n<2N}}\beta (n+h_{m})
\end{equation}
for some $\nu \; (\mathrm{mod}\; q)$, given in Section 3. The right hand side of (\ref{63}) is given by (\ref{15}). Combining this and
\[
\frac{1}{\varphi (\frac{q}{p})}=\frac{\varphi (p)}{\varphi (q)}=\frac{p-1}{\varphi (q)},
\]
we obtain
\begin{equation}
\label{64}
\begin{aligned}
S_{2, I\hspace{-.em}V }^{(m)}=& \frac{1}{\varphi (W)}\sum _{Y<p < R}(p-1)\pi ^{\flat}\left( \frac{N}{p} \right) \underset{d_{m}=e_{m}=p}{\underset{e_{1}, \cdots ,e_{k}}{\sum _{d_{1}, \cdots ,d_{k}}^{\quad\quad '}}}\frac{\lambda _{d_{1}, \cdots ,d_{k}}\lambda _{e_{1},\cdots ,e_{k}}}{\prod _{i=1}^{k}\varphi ([d_{i}, e_{i}])} \\
&\quad+O\left( \sum _{Y<p< R} \underset{d_{m}=e_{m}=p}{\underset{e_{1}, \cdots ,e_{k}}{\sum ^{\quad\quad '} _{d_{1}, \cdots ,d_{k}}}}|\lambda _{d_{1}, \cdots ,d_{k}}\lambda _{e_{1}, \cdots ,e_{k}}| \left( \left| \Delta \left( \frac{N}{p}; \frac{q}{p}, \nu _{m}^{'} \right) \right| +1 \right)       \right), 
\end{aligned}
\end{equation}
where the sum $\sum ^{\quad '}$ implies that $d_{1}, \cdots ,d_{k}, e_{1}, \cdots ,e_{k}$ are restricted to those satisfying the condition that $W, [d_{1}, e_{1}], \cdots ,[d_{k}, e_{k}]$ are pairwise coprime. The error term is, under the assumption of $BV[\theta ,{\cal P}]$, evaluated by 
\begin{equation}
\label{64.5}
\ll _{B} \frac{Ny_{\mathrm{max}}^{2}}{(\log N)^{B}}
\end{equation}
for any $B>0$. The proof is almost the same as that in Section 3, hence we omit this. Next, we compute the sum over $d_{1}, \cdots ,d_{k}, e_{1},\cdots ,e_{k}$.  By the similar way as (\ref{19}), we obtain 
\begin{equation}
\label{65}
\begin{aligned}
\underset{d_{m}=e_{m}=p}{\underset{e_{1}, \cdots ,e_{k}}{\sum _{d_{1}, \cdots ,d_{k}}^{\quad\quad '}}}\frac{\lambda _{d_{1}, \cdots ,d_{k}}\lambda _{e_{1}, \cdots ,e_{k}}}{\prod _{i=1}^{k}\varphi ([d_{i}, e_{i}])}=& \sum _{u_{1},\cdots ,u_{k}}\prod _{i=1}^{k}g(u_{i})\sum _{s_{1,2}, \cdots ,s_{k,k-1}}^{\quad\quad *}\left( \prod _{1\leq i\neq j\leq k}\mu (s_{i,j}) \right) \\
&\quad \quad \quad \quad \quad  \times \underset{d_{m}=e_{m}=p}{\underset{s_{i,j}|d_{i},e_{j} (i\neq j)}{\underset{u_{i}|d_{i}, e_{i} (\forall i)}{\underset{e_{1}, \cdots ,e_{k}}{\sum _{d_{1}, \cdots ,d_{k}}}}}}   \frac{\lambda _{d_{1}, \cdots ,d_{k}}\lambda _{e_{1}, \cdots ,e_{k}}}{\prod _{i=1}^{k}\varphi (d_{i})\varphi (e_{i})}.
\end{aligned}
\end{equation}
Using the function $y_{r_{1},\cdots ,r_{k}}^{(m)}(p)$ defined by (\ref{20}), the right hand side of (\ref{65}) is expressed by
\begin{equation}
\label{66}
\sum _{u_{1},\cdots ,u_{k}}\left( \prod _{i=1}^{k}\frac{\mu ^{2}(u_{i})}{g(u_{i})} \right)\sum _{s_{1,2}, \cdots ,s_{k,k-1}}^{\quad\quad *} \left(\prod _{1\leq i \neq j \leq k}\frac{\mu (s_{i,j})}{g(s_{i,j})^{2}} \right)y_{a_{1}, \cdots ,a_{k}}^{(m)}(p)y_{b_{1},\cdots ,b_{k}}^{(m)}(p),
\end{equation}
where $a_{i}=u_{i}\prod _{j\neq i}s_{i,j}$, $b_{i}=u_{i}\prod _{j\neq i}s_{j,i}$. Since $a_{m}=p$ implies $p=u_{m}$ or $s_{m,j}$ $(\exists j)$ and $b_{m}=p$ implies $p=u_{m}$ or $s_{j,m}$ $(\exists j)$, the contribution of the terms with $s_{i,j}\neq 1$ (hence $s_{i,j}>D_{0}$) is at most 
\begin{align*}
\sum _{\mathfrak{p}_{1}, \mathfrak{p}_{2}=1 \; or \; p} \sum _{u_{1},\cdots ,u_{k}}\prod _{i=1}^{k}\frac{\mu ^{2}(u_{i})}{g(u_{i})} & \underset{(i^{'},j^{'})\neq (i,j)}{\sum _{s_{i^{'}, j^{'}}}}  \underset{(i^{'},j^{'})\neq (i,j)}{\prod _{1\leq i^{'}\neq j^{'}\leq k}}\frac{\mu ^{2}(s_{i^{'}, j^{'}})}{g(s_{i^{'},j^{'}})^{2}} \\
 \times &\left( \sum _{s_{i,j}>D_{0}}\frac{\mu ^{2} (s_{i,j})}{g(s_{i,j})^{2}} \right)y_{a_{1},\cdots ,a_{k}}^{(m)}(p)|_{a_{m}=\mathfrak{p}_{1}}y_{b_{1},\cdots ,b_{k}}^{(m)}(p)|_{b_{m}=\mathfrak{p}_{2}}, 
\end{align*}
which is bounded by
\begin{equation}
\label{67}
\begin{aligned}
&\ll y_{\mathrm{max}}^{(m)}(p)_{r_{m}=1}^{2}\left( \frac{\varphi (W)}{W}\log R\right)^{k-1}\cdot 1 \cdot D_{0}^{-1} \\
& \quad\quad +\left( y_{\mathrm{max}}^{(m)}(p)_{r_{m}=1}y_{\mathrm{max}}^{(m)}(p)_{r_{m}=p}+y_{\mathrm{max}}^{(m)}(p)_{r_{m}=p}^{2} \right)\cdot \frac{1}{p}\left( \frac{\varphi (W)}{W}\log R \right)^{k-1}\cdot 1\cdot D_{0}^{-1} \\
&\ll \frac{\varphi (W)^{k-1}(\log R)^{k-1}}{W^{k-1}D_{0}}y_{\mathrm{max}}^{(m)}(p)_{r_{m}=1}^{2} \\
& \quad\quad +\frac{\varphi (W)^{k-1}(\log R)^{k-1}}{pW^{k-1}D_{0}}\max \{ y_{\mathrm{max}}^{(m)}(p)_{r_{m}=1},\; y_{\mathrm{max}}^{(m)}(p)_{r_{m}=p} \}^{2}.
\end{aligned}
\end{equation}
Combining (\ref{65}), (\ref{66}) and (\ref{67}), we have 
\begin{equation}
\label{68}
\begin{aligned}
\underset{d_{m}=e_{m}=p}{\underset{e_{1}, \cdots ,e_{k}}{\sum _{d_{1}, \cdots ,d_{k}}^{\quad\quad '}}}\frac{\lambda _{d_{1}, \cdots ,d_{k}}\lambda _{e_{1}, \cdots ,e_{k}}}{\prod _{i=1}^{k}\varphi ([d_{i}, e_{i}])} =& \sum _{u_{1},\cdots ,u_{k}}\frac{y_{u_{1},\cdots ,u_{k}}^{(m)}(p)^{2}}{\prod _{i=1}^{k}g(u_{i})} \\
& \quad +O\left( \frac{\varphi (W)^{k-1}(\log R)^{k-1}}{W^{k-1}D_{0}} y_{\mathrm{max}}^{(m)}(p)_{r_{m}=1}^{2}\right) \\
& \quad  +O\left( \frac{\varphi (W)^{k-1}(\log R)^{k-1}}{pW^{k-1}D_{0}}Y^{(m)}(p)^{2} \right),
\end{aligned}
\end{equation}
where 
\[
Y^{(m)}(p)=\max \{y_{\mathrm{max}}^{(m)}(p)_{r_{m}=1},\; y_{\mathrm{max}}^{(m)}(p)_{r_{m}=p} \}.
\]
By substituting  (\ref{68}) into (\ref{64}) and combining (\ref{64.5}), we obtain
\begin{equation}
\label{A}
\begin{aligned}
S_{2, I\hspace{-.em}V}^{(m)}=& \frac{1}{\varphi (W)}\sum _{Y<p< R}(p-1)\pi ^{\flat}\left( \frac{N}{p} \right) \sum _{u_{1},\cdots ,u_{k}}\frac{y_{u_{1},\cdots ,u_{k}}^{(m)}(p)^{2}}{\prod _{i=1}^{k}g(u_{i})} \\
&\quad +O\left( \frac{N\varphi (W)^{k-2}(\log N)^{k-2}}{W^{k-1}D_{0}}\sum _{Y<p< R}\left( y_{\mathrm{max}}^{(m)}(p)_{r_{m}=1}^{2}+\frac{Y^{(m)}(p)^{2}}{p}\right) \right) \\
& \quad +O_{B}\left( \frac{Ny_{\mathrm{max}}^{2}}{(\log N)^{B}} \right).
\end{aligned}
\end{equation}
Moreover, by the estimates (\ref{31}), (\ref{32}) and $\sum _{Y<p< R}1/p \ll_{\eta}1$, the first error term of (\ref{A}) is at most

\[
\frac{N\varphi (W)^{k-2}(\log N)^{k-2}}{W^{k-1}D_{0}}\cdot \frac{y_{\mathrm{max}}^{2}\varphi (W)^{2}(\log N)^{2}}{W^{2}}=\frac{y_{\mathrm{max}}^{2}N\varphi (W)^{k}(\log N)^{k}}{W^{k+1}D_{0}}.
\]
Since 
\[
(p-1) \pi ^{\flat} \left( \frac{N}{p} \right) =\frac{N}{\log \frac{N}{p}} +O\left( \frac{N}{\log ^{2}N} \right),
\]
if we replace the factor $(p-1)\pi ^{\flat}(N/p)$ in (\ref{A}) with $N/(\log N/p)$, the possible error is at most 
\begin{align*}
& \frac{1}{\varphi (W)}\cdot \frac{N}{(\log N)^{2}}\sum _{Y<p < R}\sum _{u_{1}, \cdots ,u_{k}}\frac{y_{u_{1}, \cdots ,u_{k}}^{(m)}(p)^{2}}{\prod _{i=1}^{k}g(u_{i})} \\
& \ll \frac{N}{\varphi (W)(\log N)^{2}}\sum _{Y<p < R}\left\{ \frac{y_{\mathrm{max}}^{(m)}(p)|_{r_{m}=p}^{2}}{p}+y_{\mathrm{max}}^{(m)}(p)|_{r_{m}=1}^{2} \right\}  \left( \underset{(u,W)=1}{\sum _{u< R}}\frac{1}{g(u)} \right)^{k-1} \\
& \ll \frac{N}{\varphi (W)(\log N)^{2}}\sum _{Y<p < R}\left( \frac{y_{\mathrm{max}}^{2}\varphi (W)^{2}\log ^{2}N}{pW^{2}} +\frac{y_{\mathrm{max}}^{2}\varphi (W)^{2}\log ^{2}N}{p^{2}W^{2}} \right) \\
&\quad\quad\quad\quad\quad\quad\quad\quad\quad\quad\quad\quad\quad\quad\quad\quad\quad\quad\quad\quad\quad\quad\quad\quad \times \left( \frac{\varphi (W)\log N}{W} \right)^{k-1} \\
& \ll \frac{y_{\mathrm{max}}^{2}N(\log N)^{k-1}\varphi (W)^{k}}{W^{k+1}},
\end{align*}
which is dominated by the error term above.  Therefore, 
\begin{equation}
\label{69}
\begin{aligned}
S_{2, I\hspace{-.em}V}^{(m)}=& \frac{N}{\varphi (W)}\sum _{Y<p< R}\frac{1}{\log \frac{N}{p}}\sum _{u_{1},\cdots ,u_{k}}\frac{y_{u_{1},\cdots ,u_{k}}^{(m)}(p)^{2}}{\prod _{i=1}^{k}g(u_{i})} \\
&\quad +O\left(  \frac{y_{\mathrm{max}}^{2}N\varphi (W)^{k}(\log N)^{k}}{W^{k+1}D_{0}} \right)+O_{B}\left( \frac{Ny_{\mathrm{max}}^{2}}{(\log N)^{B}} \right).
\end{aligned}
\end{equation}
Let us compute the sum over $u_{1}, \cdots ,u_{k}$ in the main term of (\ref{69}) by using Lemma 3.3. Let $u_{m}$ be $1$ or $p$. Then, by (\ref{30}), we have
\begin{align*}
&y_{u_{1}, \cdots ,u_{k}}^{(m)}(p)^{2} \\
&=\left\{ -\mu (u_{m})g(u_{m})\underset{p|a_{m}}{\sum _{a_{m}}}\frac{y_{u_{1}, \cdots ,u_{m-1},a_{m},u_{m+1}, \cdots ,u_{k}}}{\varphi (a_{m})} +O\left( \frac{y_{\mathrm{max}}g(u_{m})\varphi (W)\log \frac{R}{p}}{WD_{0}\varphi (p)} \right) \right\} ^{2} \\
&=g(u_{m})^{2}\left( \underset{p|a_{m}}{\sum _{a_{m}}}\frac{y_{u_{1}, \cdots ,u_{m-1},a_{m},u_{m+1}, \cdots ,u_{k}}}{\varphi (a_{m})} \right)^{2}+O\left( \frac{y_{\mathrm{max}}^{2}g(u_{m})^{2}\varphi (W)^{2}(\log \frac{R}{p})^{2}}{W^{2}D_{0}\varphi (p)^{2}}\right).
\end{align*}
Therefore, by taking the sum over $u_{1},\cdots ,u_{k}$, we obtain
\begin{equation}
\label{70}
\begin{aligned}
&\sum _{u_{1}, \cdots ,u_{k}}\frac{y_{u_{1}, \cdots ,u_{k}}^{(m)}(p)^{2}}{\prod _{i=1}^{k}g(u_{i})} \\
&\quad =\sum _{u_{m}=1, p}g(u_{m})\sum _{u_{1},\cdots ,u_{m-1},u_{m+1},\cdots ,u_{k}}\frac{1}{\prod _{i\neq m}g(u_{i})}\left( \underset{p|a_{m}}{\sum _{a_{m}}}\frac{y_{u_{1}, \cdots ,u_{m-1},a_{m},u_{m+1}, \cdots ,u_{k}}}{\varphi (a_{m})} \right)^{2} \\
&\quad\quad\quad +O\left( \frac{y_{\mathrm{max}}^{2}\varphi (W)^{k+1}(\log N)^{k+1}}{W^{k+1}D_{0}p }\right).
\end{aligned}
\end{equation}
By (\ref{34}), if $u_{1}, \cdots ,u_{k}$ satisfy the conditions that $(u_{i},u_{j})=1$ ($i\neq j$), $(u_{i},pW)=1$ ($\forall i\neq m)$, $u_{1}, \cdots ,u_{k}$ are square-free, we have
\begin{equation}
\label{71}
\begin{aligned}
&\left( \underset{p|a_{m}}{\sum _{a_{m}}}\frac{y_{u_{1},\cdots ,u_{m-1},a_{m},u_{m+1}, \cdots ,u_{k}}}{\varphi (a_{m})} \right)^{2} \\
&\quad\quad =\frac{\varphi (W)^{2}}{p^{2}W^{2}}\left(\log \frac{R}{p} \right)^{2}\prod _{i\neq m}\frac{\varphi (u_{i})^{2}}{u_{i}^{2}} \\
&\quad\quad\quad\quad\quad \times \left( \int _{0}^{1}F_{p}^{[m]}\left( \frac{\log u_{1}}{\log R}, \cdots ,\frac{\log u_{m-1}}{\log R}, u, \frac{\log u_{m+1}}{\log R}, \cdots ,\frac{\log u_{k}}{\log R} \right) du \right)^{2} \\
&\quad \quad \quad\quad   +O\left( \frac{\varphi (W)^{2}F_{\mathrm{max}}^{2}\log \frac{R}{p} \log \log N}{p^{2}W^{2}} \right).
\end{aligned}
\end{equation}
We substitute (\ref{71}) into (\ref{70}). Then, the contribution of the error term is at most 
\begin{align*}
& \sum _{u_{m}=1, p}g(u_{m})\left( \sum _{u_{1},\cdots, u_{m-1},u_{m+1}, \cdots ,u_{k}}\frac{1}{\prod _{i\neq m}g(u_{i})}\right)\cdot \frac{F_{\mathrm{max}}^{2}\varphi (W)^{2}\log \frac{R}{p}\log \log N}{p^{2}W^{2}} \\
& \quad \ll \frac{F_{\mathrm{max}}^{2}\varphi (W)^{k+1}(\log N)^{k}\log \log N}{pW^{k+1}}, 
\end{align*}
which is dominated by the error term of (\ref{70}). Therefore, 
\begin{equation}
\label{72}
\begin{aligned}
&\sum _{u_{1},\cdots ,u_{k}}\frac{y_{u_{1}, \cdots , u_{k}}^{(m)}(p)^{2}}{\prod _{i=1}^{k}g(u_{i})} \\
&\quad =\frac{\varphi (W)^{2}}{p^{2}W^{2}}\left( \log \frac{R}{p}\right)^{2}\sum _{u_{m}=1, p}g(u_{m}) \sum _{u_{1}, \cdots ,u_{m-1},u_{m+1},\cdots ,u_{k}}^{\quad\quad '}\prod _{i\neq m}\frac{\varphi (u_{i})^{2}}{u_{i}^{2}g(u_{i})} \\
&\quad\quad\quad\quad\quad\quad \times \left( \int _{0}^{1}F_{p}^{[m]}\left( \frac{\log u_{1}}{\log R}, \cdots ,\frac{\log u_{m-1}}{\log R}, u, \frac{\log u_{m+1}}{\log R}, \cdots ,\frac{\log u_{k}}{\log R} \right) du \right)^{2}\\
& \quad \quad +O\left( \frac{F_{\mathrm{max}}^{2}\varphi (W)^{k+1}(\log N)^{k+1}}{W^{k+1}D_{0}p} \right).
\end{aligned}
\end{equation}
The sum $\sum ^{'}$ indicates that $u_{1}, \cdots ,u_{m-1}, u_{m+1},\cdots ,u_{k}$ are restricted to those satisfying the conditions above.  In (\ref{72}), the contribution of the terms with $u_{m}=1$ are dominated by the error term. Hence  only the terms with $u_{m}=p$ contribute to the main term, and since 
\[
\frac{g(p)}{p^{2}}=\frac{1}{p}+O\left( \frac{1}{p^{2}} \right)=\frac{1}{p}+O\left( p^{-1}N^{-\eta} \right),
\]
if we replace the factor $g(p)/p^{2}$ in (\ref{72}) with $1/p$, the possible error is dominated by the error term of (\ref{72}). Hence we have 
\begin{equation}
\label{73}
\begin{aligned}
&\sum _{u_{1},\cdots ,u_{k}}\frac{y_{u_{1}, \cdots , u_{k}}^{(m)}(p)^{2}}{\prod _{i=1}^{k}g(u_{i})} \\
&\quad =\frac{\varphi (W)^{2}}{pW^{2}}\left( \log \frac{R}{p}\right)^{2} \sum _{u_{1}, \cdots ,u_{m-1},u_{m+1},\cdots ,u_{k}}^{\quad\quad '}\prod _{i\neq m}\frac{\mu ^{2}(u_{i})\varphi (u_{i})^{2}}{u_{i}^{2}g(u_{i})} \\
&\quad\quad\quad\quad\quad\quad \times \left( \int _{0}^{1}F_{p}^{[m]}\left( \frac{\log u_{1}}{\log R}, \cdots ,\frac{\log u_{m-1}}{\log R}, u, \frac{\log u_{m+1}}{\log R}, \cdots ,\frac{\log u_{k}}{\log R} \right) du \right)^{2}\\
& \quad \quad +O\left( \frac{F_{\mathrm{max}}^{2}\varphi (W)^{k+1}(\log N)^{k+1}}{W^{k+1}D_{0}p} \right).
\end{aligned}
\end{equation}
We remove the condition that $(u_{i}, u_{j})=1$ for $i \neq j$. If $(u_{i},u_{j})>1$, there exists a prime $q>D_{0}$ for which $q|u_{i}, u_{j}$. Therefore, the  difference is at most 
\begin{align*}
&\frac{\varphi (W)^{2}(\log N)^{2}F_{\mathrm{max}}^{2}}{pW^{2}}\left( \sum _{q>D_{0}}\frac{\varphi (q)^{4}}{g(q)^{2}q^{4}}\right) \left( \underset{(u,W)=1}{\sum _{u<R}}\frac{\varphi (u)^{2}}{u^{2}g(u)}\right)^{k-1} \\
& \ll \frac{F_{\mathrm{max}}^{2}\varphi (W)^{k+1}(\log N)^{k+1}}{W^{k+1}D_{0}p}.
\end{align*}
Therefore, 
\begin{equation}
\label{74}
\begin{aligned}
&\sum _{u_{1},\cdots ,u_{k}}\frac{y_{u_{1}, \cdots , u_{k}}^{(m)}(p)^{2}}{\prod _{i=1}^{k}g(u_{i})} \\
&\quad =\frac{\varphi (W)^{2}}{pW^{2}}\left( \log \frac{R}{p}\right)^{2} \underset{(u_{i},pW)=1 (\forall i\neq m)}{\sum _{u_{1}, \cdots ,u_{m-1},u_{m+1},\cdots ,u_{k}}}\prod _{i\neq m}\frac{\mu ^{2}(u_{i})\varphi (u_{i})^{2}}{u_{i}^{2}g(u_{i})} \\
&\quad\quad\quad\quad\quad\quad \times \left( \int _{0}^{1}F_{p}^{[m]}\left( \frac{\log u_{1}}{\log R}, \cdots ,\frac{\log u_{m-1}}{\log R}, u, \frac{\log u_{m+1}}{\log R}, \cdots ,\frac{\log u_{k}}{\log R} \right) du \right)^{2}\\
& \quad\quad +O\left( \frac{F_{\mathrm{max}}^{2}\varphi (W)^{k+1}(\log N)^{k+1}}{W^{k+1}D_{0}p} \right).
\end{aligned}
\end{equation}
We apply Lemma 3.4 with 
\begin{eqnarray}
\gamma (q)=\left\{ \begin{array}{ll}
 1-\frac{q^{2}-3q+1}{q^{3}-q^{2}-2q+1} & (q\mid \hspace{-.67em}/ pW) \\
 0 & (\mathrm{otherwise}) \\
\end{array} \right.
\end{eqnarray} 
to the sum over $u_{1},\cdots ,u_{m-1},u_{m+1}, \cdots ,u_{k}$. In Section 3, we proved that 
\[
L\ll \log D_{0},\quad  \mathfrak{S}=\frac{\varphi (p)\varphi (W)}{pW}(1+O(D_{0}^{-1})).
\]
Therefore, by the similar way as (\ref{42}), we find that
\begin{equation}
\label{75}
\begin{aligned}
& \underset{(u_{i},pW)=1 (\forall i\neq m)}{\sum _{u_{1}, \cdots ,u_{m-1},u_{m+1},\cdots ,u_{k}}}\prod _{i\neq m}\frac{\mu ^{2}(u_{i})\varphi (u_{i})^{2}}{u_{i}^{2}g(u_{i})}\\&\quad\quad\quad\quad\quad \times \left( \int _{0}^{1}F_{p}^{[m]}\left( \frac{\log u_{1}}{\log R}, \cdots ,\frac{\log u_{m-1}}{\log R}, u, \frac{\log u_{m+1}}{\log R}, \cdots ,\frac{\log u_{k}}{\log R} \right) du \right)^{2}\\
&\quad =\frac{\varphi (p)^{k-1}\varphi (W)^{k-1}}{p^{k-1}W^{k-1}}(\log R)^{k-1}  \\
&\quad \times \int _{0}^{1}\cdots \int _{0}^{1}\left( \int _{0}^{1}F_{p}^{[m]}(u_{1},\cdots ,u_{m-1},u,u_{m+1}, \cdots ,u_{k})du \right)^{2}du_{1}\cdots du_{m-1}du_{m+1}\cdots du_{k} \\
& \quad\quad \quad  +O\left( \frac{F_{\mathrm{max}}^{2}\varphi (p)^{k-1}\varphi (W)^{k-1}(\log N)^{k-1}}{p^{k-1}W^{k-1}D_{0}} \right).
\end{aligned}
\end{equation}
By substituting  (\ref{75}) into (\ref{74}) and replacing the factor $\varphi (p)^{k-1}/p^{k}$ with $1/p$ (the possible error is sufficiently small), we have 
\begin{equation}
\label{76}
\begin{aligned}
&\sum _{u_{1}, \cdots ,u_{k}}\frac{y_{u_{1}, \cdots ,u_{k}}^{(m)}(p)^{2}}{\prod _{i=1}^{k}g(u_{i})} \\
&\quad =\frac{ \varphi (W)^{k+1}}{pW^{k+1}}(\log R)^{k-1}\left( \log \frac{R}{p} \right) ^{2} \\
&\quad\quad \times \int _{0}^{1}\cdots \int _{0}^{1}\left( \int _{0}^{1}F_{p}^{[m]}(u_{1},\cdots ,u_{m-1},u,u_{m+1}, \cdots ,u_{k})du \right)^{2}du_{1}\cdots du_{m-1}du_{m+1}\cdots du_{k} \\
&\quad\quad +O\left( \frac{F_{\mathrm{max}}^{2}\varphi (W)^{k+1}(\log N)^{k+1}}{W^{k+1}D_{0}p} \right).
\end{aligned}
\end{equation}
We substitute this into (\ref{69}). Our next purpose is to compute the sum over $p$. For any smooth function $f$, the sum
\begin{equation}
\label{77}
\sum _{Y<p< R}\frac{1}{p}\left( \log \frac{R}{p} \right)^{2}\left( \log \frac{N}{p}\right)^{-1}f\left( \frac{\log \frac{R}{p}}{\log R}\left( \frac{\log p}{\log \frac{R}{p}}+u \right) \right) 
\end{equation}
is expressed by 
\[
\int _{N^{\eta}}^{N^{\frac{\theta}{2}-\delta}}\frac{1}{v}\left( \log \frac{R}{v} \right)^{2}\left( \log \frac{N}{v}\right)^{-1}f\left( \frac{\log \frac{R}{v}}{\log R}\left( \frac{\log v}{\log \frac{R}{v}}+u \right) \right) d\pi (v),
\]
which is asymptotically 
\[
\int _{N^{\eta}}^{N^{\frac{\theta}{2}-\delta}}\frac{1}{v}\left( \log \frac{R}{v} \right)^{2}\left( \log \frac{N}{v}\right)^{-1}f\left( \frac{\log \frac{R}{v}}{\log R}\left( \frac{\log v}{\log \frac{R}{v}}+u \right) \right) \frac{dv}{\log v}.
\]
By putting $\log v/\log N=\xi$, this becomes 
\[
\log N \int _{\eta}^{\frac{\theta}{2}-\delta}\frac{(\frac{\theta}{2}-\delta -\xi)^{2}}{1-\xi}f\left(\frac{\xi}{\frac{\theta}{2}-\delta}+\frac{\frac{\theta}{2}-\delta -\xi}{\frac{\theta}{2}-\delta}u  \right)\frac{d\xi}{\xi}.
\]
We put 
\begin{align*}
&M_{k, \delta}^{[m]}(\xi)\\
&\quad = \int _{0}^{1}\cdots \int _{0}^{1}\left( \int _{0}^{1}F_{m,\delta}(u_{1},\cdots ,u_{k};\xi)du_{m}\right)^{2}du_{1}\cdots du_{m-1}du_{m+1}\cdots du_{k},
\end{align*}
where the function $F_{m,\delta}$ is defined by (\ref{47}). Then, by applying the consequence of the  above argument to (\ref{76}), we have
\begin{equation}
\label{79}
\begin{aligned}
&\sum _{Y<p < R}\frac{1}{\log \frac{N}{p}}\sum _{u_{1},\cdots ,u_{k}}\frac{y_{u_{1},\cdots ,u_{k}}^{(m)}(p)^{2}}{\prod _{i=1}^{k}g(u_{i})} \\
&\quad =\frac{\varphi (W)^{k+1}(\log R)^{k-1}\log N}{W^{k+1}}(1+o(1))\int _{\eta}^{\frac{\theta}{2}-\delta}\frac{(\frac{\theta}{2}-\delta-\xi)^{2}}{1-\xi}M_{k,\delta}^{[m]}(\xi)\frac{d\xi}{\xi} \\
&\quad\quad +O\left( \frac{F_{\mathrm{max}}^{2}\varphi (W)^{k+1}(\log N)^{k}}{W^{k+1}D_{0}} \right),
\end{aligned}
\end{equation}
if the integral is not zero.  We substitute this into (\ref{69}).  Consequently we obtain the following result: 
\begin{prop}
Assuming $BV[\theta ,{\cal P}]$, we have
\begin{equation}
\label{80}
\begin{aligned}
S_{2, I\hspace{-.em}V}^{(m)}=&\frac{\varphi (W)^{k}N\log N(\log R)^{k-1}}{W^{k+1}}(1+o(1))M_{k,\delta}^{(m)}(F) \\
&\quad\quad +O\left( \frac{F_{\mathrm{max}}^{2}\varphi (W)^{k}N(\log N)^{k}}{W^{k+1}D_{0}} \right)+O_{B}\left( \frac{F_{\mathrm{max}}^{2}N}{(\log N)^{B}}\right)
\end{aligned}
\end{equation}
as $N\to \infty$ if 
\[
M_{k,\delta}^{(m)}(F):=\int _{\eta}^{\frac{\theta}{2}-\delta}\frac{(\frac{\theta}{2}-\delta -\xi)^{2}}{1-\xi}M_{k,\delta}^{[m]}(\xi)\frac{d\xi}{\xi}\neq 0,
\]
where
\begin{align*}
&M_{k, \delta}^{[m]}(\xi)\\
&\quad = \int _{0}^{1}\cdots \int _{0}^{1}\left( \int _{0}^{1}F_{m,\delta}(u_{1},\cdots ,u_{k};\xi)du_{m}\right)^{2}du_{1}\cdots du_{m-1}du_{m+1}\cdots du_{k},
\end{align*}
\begin{align*}
&F_{m, \delta}(u_{1}, \cdots ,u_{k}; \xi) \\
&\quad\quad\quad :=F\left( u_{1},\cdots ,u_{m-1}, \frac{\xi}{\frac{\theta}{2}-\delta}+\frac{\frac{\theta}{2}-\delta -\xi}{\frac{\theta}{2}-\delta}u_{m}, u_{m+1},\cdots ,u_{k} \right).
\end{align*}
\end{prop}
We note that if $M_{k,\delta}^{(m)}(F)=0$, the leading term vanishes and hence
\[
S_{2,I\hspace{-.em}V}^{(m)}=o \left( \frac{F_{\mathrm{max}}^{2}\varphi (W)^{k}N(\log N)^{k}}{W^{k+1}} \right)+O_{B}\left( \frac{NF_{\mathrm{max}}^{2}}{(\log N)^{B}} \right).
\]

%%%%%%%%%%%%%%%%%%%%%%%%%%%%%%%%%%%%%%%%%%%%%%%%%%%%%%%%%%%%%%%%%%%%%%%%%%%%%%%%%%%%%%%%%%%%%%%%%%%%%%%%%%%%%%%%%%%%%%%%%%%%%%%%%%%%%%%%%%%%%%%%%%%%%%%%%%%%%%%%
\section{Conclusion}
To establish the small gaps between almost primes, we consider the sum 
\begin{equation}
\label{snrho}
\begin{aligned}
S(N, \rho)=&\sum _{m=1}^{k}S_{2}^{(m)}-\rho S_{0} \\
=&\sum _{m=1}^{k}\left( S_{2,I}^{(m)}+S_{2, I\hspace{-.em}I}^{(m)}+S_{2, I\hspace{-.em}I\hspace{-.em}I}^{(m)}+S_{2, I\hspace{-.em}V}^{(m)} \right)-\rho S_{0}
\end{aligned}
\end{equation}
for $\rho \in \mathbf{N}$. To establish small gaps between the set of primes and almost primes, we consider the sum  
\begin{equation}
\label{sdnrho}
\begin{aligned}
S^{'}(N, \rho)=&\sum _{m=1}^{k}(S_{1}^{(m)}+S_{2}^{(m)})-\rho S_{0} \\
=&\sum _{m=1}^{k}\left( S_{1}^{(m)}+S_{2,I}^{(m)}+S_{2, I\hspace{-.em}I}^{(m)}+S_{2, I\hspace{-.em}I\hspace{-.em}I}^{(m)}+S_{2, I\hspace{-.em}V}^{(m)} \right)-\rho S_{0}
\end{aligned}
\end{equation}
for $\rho \in \mathbf{N}$. If $S(N, \rho)\to \infty$, there exist infinitely many $n$ for which at least $\rho +1$ of $n+h_{1}, \cdots ,n+h_{k}$ are $E_{2}$-numbers. If $S^{'}(N, \rho)\to \infty$, there exist infinitely many $n$ for which at least $\rho +1$ of $n+h_{1}, \cdots ,n+h_{k}$ are primes or $E_{2}$-numbers. We have computed all terms to obtain the asymptotic formulas for $S(N,\rho)$ and $S^{'}(N,\rho)$. The terms $S_{2, I}^{(m)}$ and $S_{2, I\hspace{-.em}I}^{(m)}$ are obtained in Proposition 3.7, the term $S_{2, I\hspace{-.em}I\hspace{-.em}I}^{(m)}$ is obtained in Proposition 4.2, and the term $S_{2, I\hspace{-.em}V}^{(m)}$ is obtained in Proposition 5.1. Also, the terms $S_{0}$ and $S_{1}^{(m)}$ are obtained by Maynard (\cite{M1}), and the result is summarized in Proposition 2.1.  By these propositions, we find that under the assumptions of $BV[\theta ,{\cal P}]$ and $BV[\theta ,{\cal E}_{2}]$, $S(N,\rho)$ is asymptotically 
\begin{align*}
&\left\{ -\theta ^{'}\sum _{m=1}^{k}L_{k, \delta }^{(m)}(F)+\frac{\theta ^{'2}}{4}\left(  \log \frac{1-\eta}{\eta} \right)\sum _{m=1}^{k}J_{k}^{(m)}(F)+\sum _{m=1}^{k}M_{k,\delta }^{(m)}(F)-\frac{\rho \theta ^{'}}{2}I_{k}(F) \right\} \\
& \quad\quad\quad \times \left( \frac{\theta ^{'}}{2} \right)^{k-1}\frac{\varphi (W)^{k}N(\log N)^{k}}{W^{k+1}}
\end{align*}
and $S^{'}(N,\rho)$ is asymptotically 
\begin{align*}
&\left\{ -\theta ^{'}\sum _{m=1}^{k}L_{k, \delta }^{(m)}(F)+\frac{\theta ^{'2}}{4}\left(1+  \log \frac{1-\eta}{\eta} \right)\sum _{m=1}^{k}J_{k}^{(m)}(F)+\sum _{m=1}^{k}M_{k,\delta }^{(m)}(F)-\frac{\rho \theta ^{'}}{2}I_{k}(F) \right\} \\
& \quad\quad\quad \times \left( \frac{\theta ^{'}}{2} \right)^{k-1}\frac{\varphi (W)^{k}N(\log N)^{k}}{W^{k+1}},
\end{align*}
whenever the leading coefficient is not zero, where 
\[
\theta ^{'}=\theta -2\delta .
\]
We take the limit $\delta \to +0$ and see when the leading coefficients 
\begin{equation}
\label{coeffofsnrho}
 -\theta \sum _{m=1}^{k}L_{k, 0 }^{(m)}(F)+\frac{\theta ^{2}}{4}\left(  \log \frac{1-\eta}{\eta} \right)\sum _{m=1}^{k}J_{k}^{(m)}(F)+\sum _{m=1}^{k}M_{k,0 }^{(m)}(F)-\frac{\rho \theta }{2}I_{k}(F) 
\end{equation}
or
\begin{equation}
\label{coeffofsdnrho}
 -\theta \sum _{m=1}^{k}L_{k, 0 }^{(m)}(F)+\frac{\theta ^{2}}{4}\left(  1+ \log \frac{1-\eta}{\eta} \right)\sum _{m=1}^{k}J_{k}^{(m)}(F)+\sum _{m=1}^{k}M_{k,0}^{(m)}(F)-\frac{\rho \theta }{2}I_{k}(F)
\end{equation}
become positive.

%%%%%%%%%%%%%%%%%%%%%%%%%%%%%%%%%%%%%%%%%%%%%%%%%%%%%%%%%%%%%%%%%%%%%%%%%%%%%%%%%%%%%%%%%%%%%%%%%%%%%%%%%%%%%%%%%%%%%%%%%%%%%%%%%%%%%%%%%%%%%%%%%%%%%%%%%%%%%%%%
\section{The proof of Theorem 1.1}
Let $\rho \in \mathbf{N}$ be sufficiently large. We use the same test function as in \cite{M1}. That is, we define the test function $F$ by 
\begin{eqnarray}
F(u_{1},\cdots ,u_{k})=\left\{ \begin{array}{ll}
 \prod _{i=1}^{k}g(ku_{i}) & (u_{1},\cdots ,u_{k}\geq 0, u_{1}+\cdots +u_{k}\leq 1) \\
 0 & (\mathrm{otherwise}), \\
\end{array} \right.
\end{eqnarray} 
where the function $g:[0, \infty) \to \mathbf{R}$ is defined by
\begin{eqnarray}
g(u)=\left\{ \begin{array}{ll}
 \frac{1}{1+Au} & (0\leq u\leq T) \\
 0 & (u >T) \\
\end{array} \right.
\end{eqnarray}
with $A=\log k-2\log \log k$, $T=(e^{A}-1)/A$.  We choose $\eta$ by 
\[
\eta =\frac{\theta T}{k}\sim \frac{\theta}{(\log k)^{3}}.
\]
In this case the function $F_{m,0}$ is given by 
\[
F_{m,0}(u_{1},\cdots ,u_{k}; \xi)=g\left( k\left( \frac{2\xi}{\theta}+\frac{\theta -2\xi}{\theta}u_{m} \right) \right)\prod _{i\neq m}g(ku_{i}).
\]
When the pair $(u_{m}, \xi)$ moves $[0,1]\times [\eta ,\frac{\theta}{2}]$, we have
\[
 k\left( \frac{2\xi}{\theta}+\frac{\theta -2\xi}{\theta}u_{m} \right)\geq \frac{2k\eta}{\theta} =2T>T.
\]
Therefore, we have $F_{m,0}\equiv 0$, so $L_{k}^{(m)}(F)=M_{k}^{(m)}(F)=0$ for $m=1,\cdots ,k$. Hence if 
\begin{equation}
\label{82}
\frac{\frac{\theta ^{2}}{4}\left( \log \frac{1-\eta}{\eta} \right)\sum _{m=1}^{k}J_{k}^{(m)}(F)}{\frac{\theta}{2}I_{k}(F)} >\rho ,
\end{equation}
(\ref{coeffofsnrho}) becomes positive. Using the inequality in \cite{M1} (p.408),  for any $\epsilon >0$,  if $k$ is sufficiently large, the left hand side of (\ref{82}) is 
\begin{equation}
\label{83}
\begin{aligned}
&=\frac{\theta}{2}\left( \log \frac{1-\eta}{\eta} \right)\frac{kJ_{k}^{(1)}(F)}{I_{k}(F)} \\
&\quad \geq \frac{3\theta}{2}( \log \log k )(1+o(1))(\log k-2\log \log k -2) \\
&\quad \geq \frac{3\theta}{2}\left( 1-\frac{\epsilon }{4}\right)(\log k\log \log k -3(\log \log k)^{2}).
\end{aligned}
\end{equation}
We put 
\[
k=\left[ \exp \left( \frac{(2+\epsilon )\rho }{3\theta \log \rho } \right)+1 \right],
\]
where $[a]$ implies the largest integer less than $a$. Then the third line of (\ref{83}) is 
\begin{align*}
& \geq \frac{3\theta}{2}\left( 1-\frac{\epsilon }{4}\right) \left( \frac{(2+\epsilon )\rho }{3\theta \log \rho } \log \left( \frac{(2+\epsilon )\rho }{3\theta \log \rho } \right) -3 \left( \log \left( \frac{(2+\epsilon )\rho }{3\theta \log \rho } \right) \right)^{2} \right) \\
& \geq \left( 1+\frac{\epsilon}{5} \right) \rho +O\left( \frac{\rho \log \log \rho}{\log \rho} \right).
\end{align*}
This is greater than $\rho$ whenever $\rho$ is sufficiently large. Hence (\ref{82}) holds for $k \sim \exp \left((2+\epsilon )\rho /(3\theta \log \rho )\right)$. We can choose the admissible set by ${\cal H}=\{ p_{\pi (k)+1}, p_{\pi (k)+2}, \cdots ,p_{\pi (k)+k} \}$, where $p_{n}$ denotes the $n$-th prime. Hence there exist infinitely many $n$ for which at least $\rho +1$ of $n+p_{\pi (k)+1}, \cdots ,n+p_{\pi (k)+k}$ are $E_{2}$- numbers. Since
\[
p_{\pi (k)+k}-p_{\pi (k)+1} \ll k\log k \ll \exp \left( \frac{(2+2\epsilon )\rho}{3\theta \log \rho } \right),
\]
by replacing $\epsilon$ with $\epsilon /2$, the proof of Theorem 1.1 is completed. \quad\quad\quad\quad\quad $\Box$ \\
\begin{rem}
Recently Takayuki Neshime, who belongs to the master course of Tokyo Institute of Technology, told me that by choosing the parameter $A$ in a different way and evaluating the contribution of $L_{k}^{(m)}(F)$, we can obtain a better upper bound
\[
\liminf _{n\to \infty} (q_{n+m}-q_{n})\ll \sqrt{m}\exp \sqrt{\frac{8m}{\theta}},\]
assuming $BV[\theta ,{\cal P}]$ and $BV[\theta ,{\cal E}_{2}]$. 
\end{rem}

%%%%%%%%%%%%%%%%%%%%%%%%%%%%%%%%%%%%%%%%%%%%%%%%%%%%%%%%%%%%%%%%%%%%%%%%%%%%%%%%%%%%%%%%%%%%%%%%%%%%%%%%%%%%%%%%%%%%%%%%%%%%%%%%%%%%%%%%%%%%%%%%%%%%%%%%%%%%%%%%
\section{The proofs of other theorems}
The numerical computations below are accomplished by {\it Mathematica}. To prove Theorem 1.2, it suffices to show that the leading coefficient (\ref{coeffofsdnrho}) with $k=6, \rho =2, \theta =\frac{1}{2}$ becomes positive for some test function $F$. We define this by 
\begin{align*}
F(x,y,z,u,v,w)=&1-\frac{143577}{50000}P_{1}+\frac{12337}{5000}P_{1}^{2}+\frac{86987}{50000}P_{2} \\
&\quad -\frac{619873}{1000000}P_{1}^{3}-\frac{156481}{100000}P_{1}P_{2}-\frac{230073}{5000000}P_{3}
\end{align*}
if $x,y,z,u,v,w \geq 0, x+y+z+u+v+w \leq 1$ and otherwise $F(x,y,z,u,v,w):=0$,  where $P_{i}=x^{i}+y^{i}+z^{i}+u^{i}+v^{i}+w^{i}$ $(i=1,2,3)$. We take $\eta =10^{-10}$. Then, by numerical computations, we find that 
\[
I_{6}(F)=5.30806\cdots \times 10^{-6}, \quad J_{6}(F)=1.88915\cdots \times 10^{-6}, 
\]
\[
L_{6, 0}^{(m)}(F)=9.20744\cdots \times 10^{-6}, \quad M_{6, 0}^{(m)}(F)=2.22265\cdots \times 10^{-6},
\]
hence 
\begin{align*}
&-\frac{1}{2}\sum _{m=1}^{6}L_{6, 0}^{(m)}(F)+\frac{1}{16}\left( 1+\log \frac{1-10^{-10}}{10^{-10}} \right)\sum _{m=1}^{6}J_{6}^{(m)}(F)+\sum _{m=1}^{6}M_{6, 0}^{(m)}(F)-\frac{1}{2}I_{6}(F) \\
&\quad  =8.02\cdots \times 10^{-8}>0.
\end{align*}
This proves Theorem 1.2. 

Next we prove Theorem 1.3. It suffices to show that the leading coefficient (\ref{coeffofsdnrho}) with $k=3, \rho =2, \theta =1$ becomes positive for some test function $F$. We define this by 
\begin{eqnarray}
F(x,y,z)=\left\{ \begin{array}{ll}
 (1-x)(1-y)(1-z) & (x,y,z \geq 0, x+y+z \leq 1) \\
 0 & (\mathrm{otherwise}) \\
\end{array} \right.
\end{eqnarray}
and put $\eta =10^{-10}$. Then, by numerical computations, we find that 
\[
I_{3}(F)=0.0287919\cdots ,\quad  J_{3}^{(m)}(F)=0.0154828 \cdots ,
\]
\[
L_{3, 0}^{(m)}(F)=0.1606331\cdots ,\quad M_{3, 0}^{(m)}(F)=0.0779163\cdots 
\]
for $m=1, 2, 3$. Consequently, 
\[
-\sum _{m=1}^{3}L_{3, 0}^{(m)}(F)+\frac{1}{4}\left( 1+\log \frac{1-10^{-10}}{10^{-10}} \right)\sum _{m=1}^{3}J_{3}^{(m)}(F)+\sum _{m=1}^{3}M_{3, 0}^{(m)}(F)-I_{3}(F) =0.00204\cdots >0.
\]
This proves Theorem 1.3. 

To prove Theorem 1.4, we see that the number  (\ref{coeffofsnrho}) with $k=5, \rho =2, \theta =1$ becomes positive for some test function $F$. We define this by 
\begin{align*}
F(x,y,z,u,v)=&1+\frac{917}{500}Q_{1}-\frac{281}{50}Q_{1}^{2}-\frac{41}{25}Q_{2}+\frac{287}{100}Q_{1}^{3} +\frac{191}{100}Q_{1}Q_{2}-\frac{81}{250}Q_{3}
\end{align*}
if $x,y,z,u,v \geq 0, x+y+z+u+v \leq 1$, and otherwise $F(x,y,z,u,v):=0$, where $Q_{i}=x^{i}+y^{i}+z^{i}+u^{i}+v^{i}$ $(i=1,2,3)$. Moreover, we take $\eta =10^{-10}$. Then by numerical computations, we find that 
\[
I_{5}(F)=\frac{1735763}{1732500000}, \quad J_{5}^{(m)}(F)=\frac{722755717}{1871100000000}, 
\]
\[
L_{5, 0}^{(m)}(F)=0.00392368 \cdots , \quad M_{5, 0}^{(m)}(F)=0.00190092 \cdots 
\]
for $1\leq m\leq 5$. Consequently, 
\begin{align*}
&-\sum _{m=1}^{5}L_{5, 0}^{(m)}(F)+\frac{1}{4}\left( \log \frac{1-10^{-10}}{10^{-10}} \right)\sum _{m=1}^{5}J_{5}^{(m)}(F)+\sum _{m=1}^{5}M_{5, 0}^{(m)}(F)-I_{5}(F) \\
&\quad \quad =2.13079\cdots \times 10^{-6} >0.
\end{align*}
Since the set ${\cal H}=\{ 0,2,6,8,12 \}$ is an admissible set with five elements, the statement of Theorem 1.4 is obtained. \quad\quad\quad\quad\quad\quad\quad\quad\quad\quad\quad\quad\quad\quad\quad\quad\quad $\Box$

%%%%%%%%%%%%%%%%%%%%%%%%%%%%%%%%%%%%%%%%%%%%%%%%%%%%%%%%%%%%%%%%%%%%%%%%%%%%%%%%%%%%%%%%%%%%%%%%%%%%%%%%%%%%%%%%%%%%%%%%%%%%%%%%%%%%%%%%%%%%%%%%%%%%%%%%%%%%%%%%%%%%%%%%%%%%%%%%%%%%%%%%%%%%%%%%%%%%%%%%%%%%%%%%%%%%%%%%%%%%%%%%%%%%%%%%%%%%%%%%

\noindent
Ehime University, \\
Dogo-Himata, Matsuyama\\
Ehime, Japan\\
E-mail address: sono.keiju.jk@ehime-u.ac.jp

\end{document}